\documentclass[a4paper,12pt]{article}
\usepackage{amsthm}
\usepackage{amsmath,amsfonts,amssymb}
\usepackage{fourier}
\usepackage{a4wide}
\usepackage[usenames,dvipsnames]{color}
\usepackage[all,cmtip]{xy}
\usepackage[verbose,colorlinks=true,linktocpage=true,linkcolor=blue,citecolor=blue]{hyperref}

\theoremstyle{definition}
\newtheorem{definition}{Definition}[section]

\theoremstyle{plain}
\newtheorem{theorem}[definition]{Theorem}
\newtheorem{proposition}[definition]{Proposition}
\newtheorem{lemma}[definition]{Lemma}
\newtheorem{corollary}[definition]{Corollary}

\theoremstyle{remark}
\newtheorem{remark}[definition]{Remark}

\numberwithin{equation}{section}

\newcommand{\Uq}{\mathrm{U}_q}
\allowdisplaybreaks[1]

\begin{document}

\title{Braided tensor products and polynomial invariants for the quantum queer superalgebra}
\author{Zhihua Chang ${}^1$ and Yongjie Wang ${}^2$\footnote{Corresponding Author: Y. Wang (Email: wyjie@mail.ustc.edu.cn)}}
\maketitle

\begin{center}
\footnotesize
\begin{itemize}
\item[1] School of Mathematics, South China University of Technology, Guangzhou, Guangdong, 510640, China.
\item[2] School of Mathematics, Hefei University of Technology, Hefei, Anhui, 230009, China.
\end{itemize}
\end{center}
\begin{abstract} 
The classical invariant theory for the queer Lie superalgebra $\mathfrak{q}_n$ investigates its invariants in the supersymmetric algebra 
$$\mathcal{U}_{s,l}^{r,k}:=\mathrm{Sym}\left(V^{\oplus r}\oplus \Pi(V)^{\oplus k}\oplus V^{*\oplus s}\oplus \Pi(V^*)^{\oplus l} \right),$$ 
where $V=\mathbb{C}^{n|n}$ is the natural supermodule, $V^*$ is its dual and $\Pi$ is the parity reversing functor. This paper aims to construct a quantum analogue $\mathcal{B}^{r,k}_{s,l}$ of $\mathcal{U}_{s,l}^{r,k}$ and to explore the quantum queer superalgebra $\Uq(\mathfrak{q}_n)$-invariants in $\mathcal{B}^{r,k}_{s,l}$.

 The strategy involves braided tensor products of the quantum analogues $\mathsf{A}_{r,n}$, $\mathsf{A}_{k,n}^{\Pi}$ of the supersymmetric algebras $\mathrm{Sym}\left(V^{\oplus r}\right)$, $\mathrm{Sym}\left(\Pi(V)^{\oplus k}\right)$, and their dual partners $\bar{\mathsf{A}}_{s,n}$, and $\bar{\mathsf{A}}_{l,n}^{\Pi}$. These braided tensor products are defined using explicit braiding operator due to the absence of a universal R-matrix for $\Uq(\mathfrak{q}_n)$. Furthermore, 
 we obtain an isomorphism between the braided tensor product $\mathsf{A}_{r,n}\otimes\mathsf{A}_{k,n}$ and $\mathsf{A}_{r+k,n}$, an isomorphism between $\mathsf{A}_{k,n}^{\Pi}$ and $\mathsf{A}_{k,n}$, as well as the corresponding isomorphisms for their dual parts. Consequently, the $\Uq(\mathfrak{q}_n)$-supermodule superalgebra $\mathcal{B}^{r,k}_{s,l}$ is identified with $\mathcal{B}^{r+k,0}_{s+l,0}$. This allows us to obtain a set of generators of $\Uq(\mathfrak{q}_n)$-invariants in $\mathcal{B}^{r,k}_{s,l}$.
\bigskip

\bigskip

\end{abstract}

\section{Introduction}
Invariant theory is one of the most inspiring themes in representation theory. It has been studied extensively for classical (super)group \cite{DLZ18, Howe95, LZ12, LZ15, LZ17, LZ21, Weyl}, Lie (super)algebra \cite{CW12, Howe95, Sergeev01}, quantum groups \cite{LZZ11, LSS22, XYZ20}, and quantum superalgebras \cite{CW23, LZZ20, Savage22, Zhangyang20}. Invariants of quantum superalgebras are not merely for mathematical interest but is also physically important.  The exploration of invariants in supersymmetric algebras under an action of a supergroup usually corresponds to observables in physical theories.

Inspired by Howe's unified approach \cite{Howe95} to investigate classical invariant theory, Sergeev obtained polynomial invariants of matrix Lie superalgebras in \cite{Sergeev01}. For the universal enveloping superalgebra $\mathrm{U}(\mathfrak{q}_n)$ of type $Q$, its invariants in the supersymmetric algebra 
$$\mathcal{U}^{r,k}_{s,l}:=\mathrm{Sym}\left(V^{\oplus r}\oplus \Pi(V)^{\oplus k}\oplus V^{*\oplus s}\oplus \Pi(V^*)^{\oplus l} \right),$$
were explicitly described, where $V$ is the natural $\mathrm{U}(\mathfrak{q}_n)$-supermodule. The $\mathrm{U}(\mathfrak{q}_n)$-supermodule $V^{\oplus r}\oplus \Pi(V)^{\oplus k}\oplus V^{*\oplus s}\oplus \Pi(V^*)^{\oplus l}$ is identified with $\left(\mathbb{C}^{r|k}\otimes V\right)\oplus\left(V^*\otimes\mathbb{C}^{s|l}\right)$. It is pointed out in \cite{Sergeev01} that $\mathbb{C}^{r|k}\otimes V$ can be further identified with $\mathbb{C}^{r+k|r+k}\circledast V$ that is one of the two isomorphic irreducible sub-$\mathrm{U}(\mathfrak{q}_{r+k})\otimes\mathrm{U}(\mathfrak{q}_n)$-supermodules of $\mathbb{C}^{r+k|r+k}\otimes V$. This leads to a $\mathrm{U}(\mathfrak{q}_{r+k})$-supermodule structure on $\mathbb{C}^{r|k}\otimes V$. Moreover, $\mathrm{Sym}\left(\mathbb{C}^{r+k|r+k}\circledast V\right)$ admits a multiplicity-free decomposition as a $\mathrm{U}(\mathfrak{q}_{r+k})\otimes\mathrm{U}(\mathfrak{q}_n)$-supermodule by the Howe duality \cite{CW12}. 
This procedure is also valid for the dual part. Thus the superalgebra $\mathcal{U}^{r,k}_{s,l}$ is identified with $\mathcal{U}^{r+k,0}_{s+l,0}$.

The purpose of this paper is to illustrate such an isomorphism in quantum case, and obtain a set of generators of the quantum queer superalgebra $\mathrm{U}_q(\mathfrak{q}_n)$-invariants in a quantum analogue  $\mathcal{B}^{r,k}_{s,l}$ of the supersymmetric algebra  $\mathcal{U}^{r,k}_{s,l}$.

A quantum deformation of $\mathrm{U}(\mathfrak{q}_n)$ was constructed in G. Olshanski's letter \cite{Olshanski92} by using Faddeev-Reshetikhin-Takhtajan (FRT) formalism. Since the queer Lie superalgebra does not have an even, non-degenerate, invariant bilinear form, the $r$-matrix $r\in\mathfrak{q}_n^{\otimes 2}$ does not satisfy the classical Yang-Baxter equation, thus the quantum queer superalgebra $\mathrm{U}_q(\mathfrak{q}_n)$ is not a quasi-triangular Hopf superalgebra.
The highest weight representation theory and their crystal basis theory  for quantum queer superalgebra were investigated in \cite{ GJKK14, GJKKK15, GJKK10}. The authors established a Howe duality  for quantum queer superalgebras by introducing a  quantum coordinate superalgebra \cite{CW20}, which is isomorphic to a braided supersymmetric algebra.  Based on the Howe duality, we established in \cite{CW23} a first fundamental theorem (FFT) of invariant theory for $\mathrm{U}_q(\mathfrak{q}_n)$ acting on a quantum analogue $\mathcal{O}_{r,s}$ of $\mathrm{Sym}\left(V^{\oplus r}\oplus V^{*\oplus s}\right)$.

Since the quantum queer superalgebra does not possess a universal $\mathcal{R}$-matrix, the technique developed in establishing  the first fundamental theorem of invariant theory for quantum groups \cite{LZZ11} could not be used in the quantum queer superpalgebra setting. Our approach to obtain the invariants for quantum queer superalgebra takes advantage of an explicit braiding operator $\Upsilon$ between a quantum coordinate superalgebra and a dual quantum coordinate superalgebra. The braiding operator enables us to define a braided tensor product, or $R$-twisted tensor products over $\mathcal{O}_{r,s}=\mathsf{A}_{r,n}\otimes\bar{\mathsf{A}}_{s,n}$ (that is, $\mathcal{B}^{r,0}_{s,0}$). More important, it is compatible with the $\mathrm{U}_q(\mathfrak{q}_n)$-action. The definition of the braiding operator is based on the $S$-matrix on the natural representation instead of the universal $\mathcal{R}$-matrix. One of the nice aspects of our approach is that all relations could be expressed via matrics, this allows us to avoid a lot of complicated calculations, see \cite{CW23} for more details.  

In this paper, we continue to utilize this advantage to investigate the $\mathrm{U}_q(\mathfrak{q}_n)$-invariants in $\mathcal{B}^{r,k}_{s,l}$. We extend the work of \cite[Theorem~1.5.1]{Sergeev01} to quantum queer superalgebra. On one hand, we explicitly construct a braiding operator $\Theta:\mathsf{A}_{k,n}\otimes\mathsf{A}_{r,n}\rightarrow\mathsf{A}_{r,n}\otimes\mathsf{A}_{k,n}$ (resp. $\bar{\Theta}:\bar{\mathsf{A}}_{l,n}\otimes\bar{\mathsf{A}}_{s,n}\rightarrow\bar{\mathsf{A}}_{s,n}\otimes\bar{\mathsf{A}}_{l,n}$),
and prove that the resulting braided tensor product of $\mathsf{A}_{r,n}\otimes\mathsf{A}_{k,n}$ (resp. $\bar{\mathsf{A}}_{s,n}\otimes\bar{\mathsf{A}}_{l,n}$) is isomorphic to $\mathsf{A}_{r+k,n}$ (resp. $\bar{\mathsf{A}}_{s+l,n}$). Thus, we obtain that $\mathsf{A}_{r,n}$ (resp. $\bar{\mathsf{A}}_{s,n}$) is isomorphic to the braided tensor product of $r$-copies of the quantum supersymmetric algebra $S_q(V)$ (resp. $S_q(V^*)$). On the other hand, we define a quantum analogue $\mathsf{A}_{k,n}^{\Pi}$ (resp. $\bar{\mathsf{A}}_{l,n}^{\Pi}$) of the supersymmetric algebra $\mathrm{Sym}(\Pi(V)^{\oplus k})$ (resp. $\mathrm{Sym}(\Pi(V^*)^{\oplus l})$ ), which is isomorphic to the $\Uq(\mathfrak{q}_n)$-supermodule superalgebra $\mathsf{A}_{k,n}$ (resp. $\bar{\mathsf{A}}_{l,n}$). This allows us to identify $\mathcal{B}^{r,k}_{s,l}=\mathsf{A}_{r,n}\otimes\mathsf{A}_{k,n}^{\Pi}\otimes\bar{\mathsf{A}}_{s,n}\otimes\bar{\mathsf{A}}_{l,n}^{\Pi}$ with $\mathcal{B}^{r+k,0}_{s+l,0}=\mathsf{A}_{r+k,n}\otimes\bar{\mathsf{A}}_{s+l,n}$. Combining with our pervious result \cite[Theorem~5.10]{CW23}, we derive generators of $\mathrm{U}_q(\mathfrak{q}_n)$-invariants in $\mathcal{B}^{r,k}_{s,l}$.

The remainder of this paper is organized as follows. We review some definitions related to braided tensor product of two superalgebras, the quantum queer superalgebra and $\mathrm{U}_q(\mathfrak{q}_n)$-supermodule superalgebra in Section~\ref{queer}.  We explicitly construct a braided tensor product $\mathsf{A}_{r,n}\otimes\mathsf{A}_{k,n}$ of two quantum coordinate superalgebras and identify it with $\mathsf{A}_{r+k,n}$ in Section~\ref{sec:qcs}. The definition of a quantum analogue $\mathsf{A}_{k,n}^{\Pi}$ of $\mathrm{Sym}\left(\Pi(V)^{\oplus k}\right)$ and an isomorphism from $\mathsf{A}_{k,n}^{\Pi}$ to $\mathsf{A}_{k,n}$ are also obtained  in this section. Section~\ref{sec:dqcsa} serves for discussing braided tensor products of dual quantum coordinate superalgebras and a quantum analogue of   $\mathrm{Sym}\left(\Pi(V^*)^{\oplus l}\right)$.  The $\Uq(\mathfrak{q}_n)$-invariants in a quantum analogue $\mathcal{B}^{r,k}_{s,l}$ of the supersymmetric algebra $\mathcal{U}^{r,k}_{s,l}$ are investigated  in Section~\ref{sec:invariant}.

\section{Quantum queer superalgebra and their coordinate superalgebra}
\label{queer}
Let $A$ and $B$ be two superalgebras over a commutative ring. A twisted tensor product {\cite[Definition~2.1]{CSV95}} of $A$ and $B$ is a superalgebra $C$, together with two injective homomorphisms $i_A: A\rightarrow C$ and $i_B: B\rightarrow C$, such that the canonical linear map $(i_A, i_B): A\otimes B\rightarrow C$ defined by $(i_A,i_B)(a\otimes b)=i_A(a)i_B(b)$ is a linear superalgebra isomorphism. 

Twisted tensor products can be characterized by twisting maps. According to \cite[Proposition~2.7]{CSV95}, for any twisted tensor product $(C, i_A,i_B)$ of two superalgebras $A$ and $B$, there exists a twisting map $\sigma: B\otimes A\rightarrow A\otimes B$ such that $C$ is isomorphic to $A\otimes B$ as superalgebras under the following multiplication $\mu_\sigma$
$$\mu_\sigma:=(\mu_A\otimes \mu_B)(\text{id}_A\otimes \sigma\otimes\text{id}_B),$$
where $\mu_A$ (resp. $\mu_B$) is the multiplication on $A$ (resp. $B$). Furthermore, if $\sigma$ satisfies the usual hexagon axioms, then $\sigma$ is a braiding operator in sense of \cite[Definition~8.1.1]{EGNO15}.  The twisted tensor product $A\otimes B$ is also called a braided tensor product.

For the quantum group $\mathrm{U}_q(\mathfrak{gl}_n)$ of type $A$, a braiding operator  of $\mathrm{U}_q(\mathfrak{gl}_n)$-module algebras is given by the composition of the permutation operator $P$ with the universal $\mathcal{R}$-matrix. However, the quantum superalgebra of type $Q$ does not have a universal $\mathcal{R}$-matrix as pointed out in \cite{Olshanski92}. In this paper,  we aim to construct an explicit braiding operator among quantum coordinate superalgebras and dual quantum coordinate superalgebras, see Propositions \ref{prop:interwa} and \ref{prop:interbara}.

We always assume that the base field $\mathbb{C}(q)$ is the field of rational functions in an indeterminate $q$. For a positive integer $n$, we denote $I_{n|n}:=\big\{-n,\ldots,-1,1,\ldots,n\big\}$. Let $V_q$ be the $2n$-dimensional $\mathbb{C}(q)$-vector superspace with the basis $\{v_i, i\in I_{n|n}\}$, which is equipped with a $\mathbb{Z}/2\mathbb{Z}$-grading
$$|v_i|=|i|:=\begin{cases}\bar{0},&\text{if }i>0,\\ \bar{1}&\text{if }i<0.\end{cases}$$
Then $\mathrm{End}(V_q)$ is naturally an associative superalgebra, in which the standard matrix unit $E_{ij}$ is of parity $|i|+|j|$ for $i,j\in I_{n|n}$. 

As in \cite{Olshanski92}, we set
\begin{align}
S:=&\sum\limits_{i,j\in I_{n|n}}q^{\varphi(i,j)}E_{ii}\otimes E_{jj}+\xi \sum\limits_{i<j }(-1)^{|i|}(E_{ji}+E_{-j,-i})\otimes E_{ij}\in\mathrm{End}(V_q)^{\otimes 2}\label{eq:smatrix}\\
=&\sum\limits_{i,j\in I_{n|n}}S_{ij}\otimes E_{ij},\nonumber
\end{align}
where $\delta_{ij}$ is standard Kronecker symbol, $\varphi(i,j)=(-1)^{|j|}(\delta_{ij}+\delta_{i,-j})$ and $\xi=q-q^{-1}$. It satisfies the quantum Yang-Baxter equation:
$$S^{12}S^{13}S^{23}=S^{23}S^{13}S^{12},$$
where 
$$S^{12}=S\otimes 1,\quad S^{23}=1\otimes S,\quad S^{13}=\sum\limits_{i,j\in I_{n|n}}S_{ij}\otimes 1\otimes E_{ij}.$$

\textit{The quantum queer superalgebra} is defined via Faddeev-Reshetikhin-Takhtajan presentation as follows:

\begin{definition}[{G. Olshanski \cite[Definition~4.2]{Olshanski92}}]
The quantum queer superalgebra $\Uq(\mathfrak{q}_n)$ is the unital associative superalgebra over $\mathbb{C}(q)$ generated by elements $L_{ij}$ of parity $|i|+|j|$ for $i,j\in I_{n|n}$ and $i\leqslant j$, with defining relations:
\begin{eqnarray}
&L_{ii}L_{-i,-i}=1=L_{-i,-i}L_{ii},&\\
&L^{[1]2}L^{[1]3}S^{23}=S^{23}L^{[1]3}L^{[1]2},&\label{eq:SLL}
\end{eqnarray}
where $L^{[1]2}=\sum\limits_{i\leqslant j} L_{ij}\otimes E_{ij}\otimes 1$, $L^{[1]3}=\sum\limits_{i\leqslant j} L_{ij}\otimes1\otimes E_{ij}$ and the relation \eqref{eq:SLL} holds in $\Uq(\mathfrak{q}_n)\otimes\mathrm{End}(V_q)\otimes\mathrm{End}(V_q)$. 
\end{definition}

The associative superalgebra $\Uq(\mathfrak{q}_n)$ is a Hopf superalgebra with the comultiplication $\Delta$, the counit $\varepsilon$  and the antipode $\mathcal{S}$ given by:
\begin{equation*}
\Delta(L)=L\otimes L,\qquad \varepsilon(L)=1,\qquad \mathcal{S}(L)=L^{-1}.
\end{equation*}

The $\mathbb{C}(q)$-vector superspace $V_q$ is naturally a $\Uq(\mathfrak{q}_n)$-supermodule via the homomorphism 
$$\Uq(\mathfrak{q}_n)\rightarrow\mathrm{End}(V_q), \quad L\mapsto S,$$
where $S$ is the matrix given in \eqref{eq:smatrix}.

If $B$ is a superalgebra which is also a $\Uq(\mathfrak{q}_n)$-supermodule, we will consider the compatibility of the $\Uq(\mathfrak{q}_n)$-action and the multiplication on $B$ in the following sense:

\begin{definition}
\label{def:smsa}
Let $B$ be a unital associative superalgebra that is also a $\Uq(\mathfrak{q}_n)$-supermodule.
\begin{enumerate}
\item $B$ is a $\mathrm{U}_q(\mathfrak{q}_n)$-supermodule superalgebra if 
$$u.(ab)=\sum\limits_{(u)}(-1)^{|u_{(2)}||a|}(u_{(1)}.a) (u_{(2)}.b),\text{ and } u.1=\varepsilon(u)1,$$
for all $a,b\in B$ and $u\in\mathrm{U}_q(\mathfrak{q}_n)$, where we use the Sweedler's notation  $\Delta(u)=\sum_{(u)}u_{(1)}\otimes u_{(2)}$.
\item $B$ is a $\mathrm{U}_q(\mathfrak{q}_n)^{\mathrm{cop}}$-supermodule superalgebra if 
$$u.(ab)=\sum\limits_{(u)}(-1)^{|u_{(1)}|(|u_{(2)}|+|a|)}(u_{(2)}.a) (u_{(1)}.b),\text{ and } u.1=\varepsilon(u)1,$$
for all $a,b\in B$ and $u\in\mathrm{U}_q(\mathfrak{q}_n)$.
\end{enumerate}
\end{definition}

Let $B_1$ and $B_2$ be two $\Uq(\mathfrak{q}_n)$-supermodule superalgebras. Suppose that there is a homomorphism of $\Uq(\mathfrak{q}_n)$-supermodules $\vartheta: B_2\otimes B_1\rightarrow B_1\otimes B_2$ such that the following diagrams commute
	\begin{equation}
		\xymatrix{
			B_2\otimes B_1\otimes B_1
			\ar[d]_{1\otimes\mathrm{mul}} \ar[r]^{\vartheta\otimes1}
			&B_1\otimes B_2\otimes B_1
			\ar[r]^{1\otimes\vartheta}
			&B_1\otimes B_1\otimes
			B_2
			\ar[d]^{\mathrm{mul}\otimes1}\\
			B_2\otimes B_1
			\ar[rr]_{\vartheta}
			&&B_1\otimes B_2
		}
		\label{eq:intwass1}
	\end{equation}
	and
	\begin{equation}
		\xymatrix{
			B_2\otimes B_2\otimes B_1
			\ar[r]^{1\otimes\vartheta}\ar[d]_{\mathrm{mul}\otimes1}
			&B_2\otimes B_1\otimes B_2
			\ar[r]^{\Theta\otimes1}
			&B_1\otimes B_2
			\otimes B_2
			\ar[d]^{1\otimes\mathrm{mul}}\\
		B_2\otimes B_1
			\ar[rr]_{\vartheta}
			&&B_1\otimes B_2
		}\label{eq:intwass2}
	\end{equation}
	where $\mathrm{mul}$ denotes the multiplication map. Then it has been shown in \cite[Lemma~3.1]{CW23} that the braided tensor product $B_1\otimes B_2$ is also a $\Uq(\mathfrak{q}_n)$-supermodule superalgebra under the multiplication
$$B_1\otimes B_2\otimes B_1\otimes B_2\xrightarrow{1\otimes\vartheta\otimes1} B_1\otimes B_1\otimes B_2\otimes B_2\xrightarrow{\mathrm{mul}\otimes\mathrm{mul}} B_1\otimes B_2.$$

\section{Quantum Coordinate Superalgebras}\label{sec:qcs}

A quantum coordinate superalgebra $\mathsf{A}_{r,n}$ of type $Q$ is a quantum analogue of the supersymmetric algebra $\mathrm{Sym}\left(V^{\oplus r}\right)$ on which $\mathrm{U}(\mathfrak{q}_n)$ acts as superalgebra endmorphism. In the classical (non-quantum) case, there is an isomorphism of $\mathrm{U}(\mathfrak{q}_n)$-supermodule superalgebras
$$\mathrm{Sym}\left(V^{\oplus (r+k)}\right)\cong
\mathrm{Sym}\left(V^{\oplus r}\right)\otimes\mathrm{Sym}\left(V^{\oplus k}\right).$$
This section is devoted to establishing such an isomorphism in quantum case where the superalgebra structure on the tensor product of two quantum coordinate superalgebras will be the braided tensor product.

The quantum coordinate superalgebra $\mathsf{A}_{r,n}$ is defined in \cite[Section~4]{CW20} as a sub-superalgebra of the finite dual $\Uq(\mathfrak{q}_n)^{\circ}$ generated by certain matrix elements with respect to the natural $\Uq(\mathfrak{q}_n)$-supermodule $V_q$. According to \cite[Proposition~3.6.4]{BDK20} and \cite[Section~2]{CW23}, the quantum coordinate superalgebra $\mathsf{A}_{r,n}$ can also be described via generators and relations as follows:

\begin{definition}
The \textit{quantum coordinate superalgebra} $\mathsf{A}_{r,n}$ is the unital associative superalgebra generated by $t_{ia}$ of parity $|i|+|a|$ with $i\in I_{r|r}, a\in I_{n|n}$, subjects to the relations
\begin{align}
t_{ia}&=t_{-i,-a},\label{eq:QCA1}\\
S^{12}T^{1[3]}T^{2[3]}&=T^{2[3]}T^{1[3]}S^{12},\label{eq:QCA2}
\end{align}
where $T=\sum\limits_{i\in I_{r|r}, a\in I_{n|n}}E_{ia}\otimes t_{ia}$.  It is a quantum analogue of the supersymmetric algebra $\mathrm{Sym}\left(V^{\oplus r}\right)$.
\end{definition}

It is shown in \cite[Lemma~3.2]{CW23} that the superalgebra $\mathsf{A}_{r,n}$ is also presented by generators $t_{ia}$ with $i=1,\ldots,r,  a\in I_{n|n}$ and the relation
\begin{equation}
R_+^{12}T_+^{1[3]}T_+^{2[3]}=T_+^{2[3]}T_+^{1[3]}S^{12},\label{eq:ArlnsM}
\end{equation}
where $T_+=\sum\limits_{i=1}^{r}\sum\limits_{a\in I_{n|n}}E_{ia}\otimes t_{ia}$, $S$ is the matrix \eqref{eq:smatrix}, and 
\begin{equation}
R_+:=\sum\limits_{i,j=1}^rq^{\delta_{ij}}E_{ii}\otimes E_{jj}+\xi \sum\limits_{1\leqslant i<j\leqslant r }E_{ji}\otimes E_{ij}\,\label{eq:rmatrix}
\end{equation}
is the submatrix\footnote{The submatrix $R$ of $S$ is exactly the R-matrix of $\Uq(\mathfrak{gl}_r)$} of $S$ involving the terms $E_{ik}\otimes E_{jl}$ with $1\leqslant i, j,k,l\leqslant r$. 

Beside being an associative superalgebra, $\mathsf{A}_{r,n}$ is also a $\Uq(\mathfrak{q}_n)$-supermodule superalgebra (see \cite[Lemma~2.3]{CW23}) in the sense of Definition~\ref{def:smsa} with the action $\Phi$ determined by
\begin{equation}
L^{[2]3}\underset{\Phi}{\cdot}T^{1[2]}=T^{1[2]}S^{13},
\label{eq:QCAUq}
\end{equation}
where $S$ is the matrix \eqref{eq:smatrix} and $$L^{[2]3}\underset{\Phi}{\cdot}T^{1[2]}:=\sum\limits_{a,b,c\in I_{n|n}}\sum\limits_{i\in I_{r|r}}(-1)^{(|a|+|b|)(|i|+|c|)}E_{ic}\otimes\Phi_{L_{ab}}(t_{ic})\otimes E_{ab}.$$
According to \cite[Remark~4.4]{CW20}, the $\Uq(\mathfrak{q}_n)$-supermodule superalgebra $\mathsf{A}_{r,n}$ is a quantum analogue of the $\mathrm{U}(\mathfrak{q}_n)$-supermodule superalgebra $\mathrm{Sym}\left(V^{\oplus r}\right)$.

\begin{lemma}
\label{lem:embed}
For each $0\leq p\leq k$, there is an injective homomorphism of $\Uq(\mathfrak{q}_n)$-supermodule superalgebras
$$\iota_p: \mathsf{A}_{r,n}\rightarrow \mathsf{A}_{r+k,n},\quad
t_{ia}\mapsto t_{p+i,a},\quad i=1,\ldots,r,\quad a\in I_{n|n}.$$
\end{lemma}
\begin{proof}
Since $\mathsf{A}_{r,n}$ is generated by $t_{ia}, i=1,\ldots,r, a\in I_{n|n}$, the assignment $t_{ia}\mapsto t_{p+i,a}$ entends to a homomorphism from the free unital associative superalgebra generated by $t_{ia}, i=1,\ldots,r, a\in I_{n|n}$ to the associative superalgebra $\mathsf{A}_{r+k,n}$. 

By \eqref{eq:ArlnsM}, the elements $t_{ia}$ with $i=1,\ldots, r, a\in I_{n|n}$ in $\mathsf{A}_{r,n}$ satisfy the relation: 
 \begin{equation*}
	\begin{aligned}
		q^{\delta_{ij}}t_{ia}t_{jb}-(-1)^{|a||b|}q^{\varphi(a,b)}t_{jb}t_{ia}
		=\xi (\delta_{a<b}-\delta_{j<i})t_{ja}t_{ib}
		+(-1)^{|b|}\xi\delta_{-a<b}t_{j,-a}t_{i,-b}.
	\end{aligned}
\end{equation*}
While the elements $t_{p+i,a}$ with $i=1,\ldots, r, a\in I_{n|n}$ in $\mathsf{A}_{r,n}$ in $\mathsf{A}_{r+k, n}$ satisfy the relation:
 \begin{equation*}
	\begin{aligned}
		&q^{\delta_{p+i,p+j}}t_{p+i,a}t_{p+j,b}-(-1)^{|a||b|}q^{\varphi(a,b)}t_{p+j,b}t_{p+i,a}\\
		=&\xi (\delta_{a<b}-\delta_{p+j<p+i})t_{p+j,a}t_{p+i,b}
		+(-1)^{|b|}\xi\delta_{-a<b}t_{p+j,-a}t_{p+i,-b}.
	\end{aligned}
\end{equation*}
Hence, there is a well-defined homomorphism of associative superalgebras $\iota_p:\mathsf{A}_{r,n}\rightarrow\mathsf{A}_{r+k,n}$ such that $t_{ia}\mapsto t_{p+i, a}$.
Note that the $\Uq(\mathfrak{q}_n)$-actions on $\mathsf{A}_{r,n}$ and $\mathsf{A}_{r+k,n}$ are given by \eqref{eq:QCAUq}, $\iota_p$ is also a homomorphism of $\Uq(\mathfrak{q}_n)$-supermodules.

The injectivity of $\iota_p$ follows from \cite[Proposition~3.6.4]{BDK20}. It states that the monomials
$$\prod\limits_{i=1,\ldots,r, a\in I_{n|n}}t_{ia}^{d_{ia}}, \quad d_{ia}\geq0\text{ if }|a|=\bar{0}\text{ and }d_{ia}\in{0,1}\text{ if }|a|=\bar{1}$$
form a $\mathbb{C}(q)$-basis of $\mathsf{A}_{r,n}$, where the product is taken with respect to the lexicographic order.
\end{proof}

In order to properly define a braided multiplication on the tensor product $\mathsf{A}_{r,n}\otimes\mathsf{A}_{k,n}$ such that it is compatible with the $\Uq(\mathfrak{q}_n)$-action on $\mathsf{A}_{r,n}\otimes\mathsf{A}_{k,n}$ in the sense of Definition~\ref{def:smsa}, we need a braiding operator defined in the following proposition.

\begin{proposition}\label{prop:interwa}
	There is a unique homomorphism of $\Uq(\mathfrak{q}_n)$-supermodules
	$$\Theta:\mathsf{A}_{k,n}\otimes\mathsf{A}_{r,n}
	\rightarrow\mathsf{A}_{r,n}\otimes\mathsf{A}_{k,n}$$
	satisfying the commutative diagrams \eqref{eq:intwass1},\eqref{eq:intwass2}, and such that 
	\begin{align}
		\Theta\left(1\otimes x\right)=&x\otimes1,\quad x\in\mathsf{A}_{r,n}\label{eq:theta1a}\\
        \Theta\left(y\otimes1\right)=&1\otimes y\quad y\in\mathsf{A}_{k,n},\label{eq:theta1b}\\		
		\Theta\left(T_k^{1[3]}T_r^{2[4]}\right)
		=&T_r^{2[3]}T_k^{1[4]}S^{12},\label{eq:theta2}
	\end{align}
where $T_r=\sum\limits_{i=1,\ldots,r \atop a\in I_{n|n}}E_{ia}\otimes t_{ia}$ and $T_k=\sum\limits_{j=1,\ldots,k\atop b\in I_{n|n}}E_{jb}\otimes t_{jb}$ are the generator matrices of $\mathsf{A}_{r,n}$ and $\mathsf{A}_{k,n}$ respectively.
\end{proposition}

\begin{proof}
Let $W_r$ (resp. $W_k$)be the $\mathbb{C}(q)$-sub-superspace of $\mathsf{A}_{r,n}$ (resp. $\mathsf{A}_{k,n}$) spanned by $1$ and $t_{ia}$ for $i=1,\ldots,r$ and $a\in I_{n|n}$ (resp. $t_{ia}$ for $i=1,\ldots,k$ and $a\in I_{n|n}$). Then \eqref{eq:theta1a}, \eqref{eq:theta1b} and \eqref{eq:theta2} define a $\mathbb{C}(q)$-linear map $W_k\otimes W_r\rightarrow W_r\otimes W_k$, which extends to a $\mathbb{C}(q)$-linear map $\widetilde{\Theta}:\mathsf{F}_{k,n}\otimes\mathsf{F}_{r,n}
	\rightarrow\mathsf{A}_{r,n}\otimes\mathsf{A}_{k,n}$ satisfying the commutative diagrams \eqref{eq:intwass1} and \eqref{eq:intwass2}, where $\mathsf{F}_{r,n}$ and $\mathsf{F}_{k,n}$ are the free associative superalgebras generated by $t_{ia}$, $i=1,\ldots, r$, $a\in I_{n|n}$ and $t_{jb}$, $j=1,\ldots, k$, $b\in I_{n|n}$, respectively.

We need to show that $\widetilde{\Theta}$ preserves the defining relations \eqref{eq:ArlnsM} for $\mathsf{A}_{r,n}$ and $\mathsf{A}_{k,n}$. Recall from \eqref{eq:ArlnsM} that the superalgebra $\mathsf{A}_{r,n}$ is presented by the generators $t_{ia}$, $i=1,\ldots,r$ and $a\in I_{n|n}$ and the relation
$$R_+^{12}T_r^{1[3]}T_r^{2[3]}=T_r^{2[3]}T_r^{1[3]}S^{12}.$$
Using the commutative diagram \eqref{eq:intwass1}, we verify that
\begin{align*}
\widetilde{\Theta}\left(T_k^{1[4]}T_r^{2[5]}T_r^{3[5]}\right)
=&
\widetilde{\Theta}\circ(1\otimes\mathrm{mul})\left(T_k^{1[4]}T_r^{2[5]}T_r^{3[6]}\right)\\
=&(\mathrm{mul}\otimes1)\circ(1\otimes\widetilde{\Theta})\circ(\widetilde{\Theta}\otimes1)
\left(T_k^{1[4]}T_r^{2[5]}T_r^{3[6]}\right)\\
=&(\mathrm{mul}\otimes1)\circ(1\otimes\widetilde{\Theta})
\left(T_r^{2[4]}T_k^{1[5]}S^{12}T_r^{3[6]}\right)\\
=&(\mathrm{mul}\otimes1)
\left(T_r^{2[4]}\widetilde{\Theta}\left(T_k^{1[5]}T_r^{3[6]}\right)S^{12}\right)\\
=&(\mathrm{mul}\otimes1)
\left(T_r^{2[4]}T_r^{3[5]}T_k^{1[6]}S^{13}S^{12}\right)\\
=&T_r^{2[4]}T_r^{3[4]}T_k^{1[5]}S^{13}S^{12},\\
\widetilde{\Theta}\left(T_k^{1[4]}T_r^{3[5]}T_r^{2[5]}\right)
=&\widetilde{\Theta}\circ(1\otimes\mathrm{mul})\left(T_k^{1[4]}T_r^{3[5]}T_r^{2[6]}\right)\\
=&(\mathrm{mul}\otimes1)\circ(1\otimes \widetilde{\Theta})\circ(\widetilde{\Theta}\otimes1)
\left(T_k^{1[4]}T_r^{3[5]}T_r^{2[6]}\right)\\
=&(\mathrm{mul}\otimes1)\circ(1\otimes \widetilde{\Theta})
\left(T_r^{3[4]}T_k^{1[5]}S^{13}T_r^{2[6]}\right)\\
=&(\mathrm{mul}\otimes1)\left(T_r^{3[4]}T_r^{2[5]}T_k^{1[6]}S^{12}S^{13}\right)\\
=&T_r^{3[4]}T_r^{2[4]}T_k^{1[5]}S^{12}S^{13},
\end{align*}
which yield that
\begin{align*}
\widetilde{\Theta}\left(T_k^{1[4]}R_+^{23}T_r^{2[5]}T_r^{3[5]}\right)
=&R_+^{23}\widetilde{\Theta}\left(T_k^{1[4]}T_r^{2[5]}T_r^{3[5]}\right)
=R_+^{23}T_r^{2[4]}T_r^{3[4]}T_k^{1[5]}S^{13}S^{12}\\
=&T_r^{3[4]}T_r^{2[4]}T_k^{1[5]}S^{23}S^{13}S^{12},\\
\widetilde{\Theta}\left(T_k^{1[4]}T_r^{3[5]}T_r^{2[5]}S^{23}\right)=&T_r^{3[4]}T_r^{2[4]}T_k^{1[5]}S^{12}S^{13}S^{23}.
\end{align*}
Note that the matrix $S$ satisfies the quantum Yang-Baxter equation
$$S^{12}S^{13}S^{23}=S^{23}S^{13}S^{12}.$$
Thus, we obtain that
$$\widetilde{\Theta}\left(T_k^{1[4]}R_+^{23}T_r^{2[5]}T_r^{3[5]}\right)
=\widetilde{\Theta}\left(T_k^{1[4]}T_r^{3[5]}T_r^{2[5]}S^{23}\right).$$

Similarly, we use the commutative diagram \eqref{eq:intwass2} to deduce that 
\begin{align*}
\widetilde{\Theta}\left(R_+^{12}T_k^{1[4]}T_k^{2[4]}T_r^{3[5]}\right)
=\widetilde{\Theta}\left(T_k^{2[4]}T_k^{1[4]}S^{12}T_r^{3[5]}\right).
\end{align*}
Hence, $\widetilde{\Theta}$ induces a well-defined $\mathbb{C}(q)$-linear map $\Theta:\mathsf{A}_{k,n}\otimes\mathsf{A}_{r,n}\rightarrow\mathsf{A}_{r,n}\otimes\mathsf{A}_{k,n}$ which satisfies \eqref{eq:theta1a}, \eqref{eq:theta1b}, \eqref{eq:theta2} and the commutative diagrams \eqref{eq:intwass1}, \eqref{eq:intwass2}. Such a $\mathbb{C}(q)$-linear map is unique since its image $\Theta(y\otimes x)$ is determined by \eqref{eq:theta1a}, \eqref{eq:theta1b} and \eqref{eq:theta2} when $x$ is one of the generators $t_{ia}$, $i=1,\ldots, r$, $a\in I_{n|n}$ of $\mathsf{A}_{r,n}$ and $y$ is one of the generators $t_{jb}$, $j=1,\ldots, k$, $a\in I_{n|n}$. 

Finally, we show that $\Theta$ is also a homomorphism of $\Uq(\mathfrak{q}_n)$-supermodules. Since both $\mathsf{A}_{r,n}$ and $\mathsf{A}_{k,n}$ are $\Uq(\mathfrak{q}_n)$-supermodule superalgebras, it suffices to verify that $\Theta$ commutes with the action of $\Uq(\mathfrak{q}_n)$ on generators. Since $\Uq(\mathfrak{q}_n)$ acts on $\mathsf{A}_{k,n}\otimes\mathsf{A}_{r,n}$ and $\mathsf{A}_{r,n}\otimes\mathsf{A}_{k,n}$ via the comultiplication
$$\Delta\left(L^{[1]2}\right)=L^{[1]2}L^{[1']2},$$ 
we verify that
\begin{align*}
\Theta\left(L^{[3]4}\underset{\Phi}{\cdot}
\left(T_k^{1[3]}T_r^{2[3']}\right)\right)
=&\Theta\left(\left(L^{[3]4}\underset{\Phi}{\cdot}T_k^{1[3]}\right)\left(L^{[3']4}\underset{\Phi}{\cdot}T_r^{2[3']}\right)\right)
=\Theta\left(T_k^{1[3]}S^{14}T_r^{2[3']}S^{24}\right)\\
=&T_r^{2[3]}T_k^{1[3']}S^{12}S^{14}S^{24}.\\
L^{[3]4}\underset{\Phi}{\cdot}\Theta\left(T_k^{1[3]}T_r^{2[3']}\right)
=&L^{[3]4}\underset{\Phi}{\cdot}\left(T_r^{2[3]}T_k^{1[3']}\right)S^{12}
=\left(L^{[3]4}\underset{\Phi}{\cdot}T_r^{2[3]}\right)\left(L^{[3]4}\underset{\Phi}{\cdot}T_k^{1[3']}\right)S^{12}\\
=&T_r^{2[3]}S^{24}T_k^{1[3']}S^{14}S^{12}
=T_r^{2[3]}T_k^{1[3']}S^{24}S^{14}S^{12}.
\end{align*}
Thus,
$$\Theta\left(L^{[3]4}\underset{\Phi}{\cdot}\left(T_k^{1[3]}T_r^{2[3']}\right)\right)
=L^{[3]4}\underset{\Phi}{\cdot}\Theta\left(T_k^{1[3]}T_r^{2[3']}\right),$$
which proves that $\Theta$ is a homomorphism of $\Uq(\mathfrak{q}_n)$-supermodules.
\end{proof}

Now, we obtain a $\Uq(\mathfrak{q}_n)$-supermodule homomorphism $\Theta:\mathsf{A}_{k,n}\otimes\mathsf{A}_{r,n}\rightarrow\mathsf{A}_{r,n}\otimes\mathsf{A}_{k,n}$ satisfying the commutative diagrams \eqref{eq:intwass1} and \eqref{eq:intwass2}. Hence, $\mathsf{A}_{r,n}\otimes\mathsf{A}_{k,n}$ is equipped with a braided multiplication:
$$\mathsf{A}_{r,n}\otimes\mathsf{A}_{k,n}
\otimes\mathsf{A}_{r,n}\otimes\mathsf{A}_{k,n}
\xrightarrow{1\otimes\Theta\otimes1}\mathsf{A}_{r,n}\otimes\mathsf{A}_{r,n}\otimes\mathsf{A}_{k,n}\otimes\mathsf{A}_{k,n}\xrightarrow{\mathrm{mul}\otimes\mathrm{mul}}
\mathsf{A}_{r,n}\otimes\mathsf{A}_{k,n}.
$$
By \cite[Lemma~3.1]{CW23}, $\mathsf{A}_{r,n}\otimes\mathsf{A}_{k,n}$ is a $\Uq(\mathfrak{q}_n)$-supermodule superalgebra. In terms of generators, we have
\begin{equation*}
	\left(1\otimes t_{jb}\right)\left(t_{ia}\otimes1\right)=(-1)^{|a||b|}q^{\varphi(b,a)}t_{ia}\otimes t_{jb}
+\delta_{b<a}\xi t_{ib}\otimes t_{ja}+\delta_{-b<a}(-1)^{|a|}\xi t_{i,-b}\otimes t_{j,-a},
\end{equation*}
for $i=1,\ldots,r$, $j=1,\ldots,k$ and $a,b\in I_{n|n}$.

\begin{lemma}
	\label{lem:braidmul}
Let $x\in\mathsf{A}_{r,n}$ and $y\in\mathsf{A}_{k,n}$. Suppose that
$$\Theta(y\otimes x)=\sum\limits_i x_i\otimes y_i,$$
where $x_i\in\mathsf{A}_{r,n}$, $y_i\in\mathsf{A}_{k,n}$ and the summation is taken over a finite set. Then the following equality holds in $\mathsf{A}_{r+k,n}$
\begin{equation*}
\iota_0(y)\iota_k(x)=\sum\limits_i\iota_k(x_i)\iota_0(y_i),
\end{equation*}
where $\iota_0:\mathsf{A}_{k,n}\rightarrow\mathsf{A}_{r+k,n}$ and $\iota_k:\mathsf{A}_{r,n}\rightarrow\mathsf{A}_{r+k,n}$ are the maps given in Lemma~\ref{lem:embed}.
\end{lemma}
\begin{proof}
The superalgebra $\mathsf{A}_{r,n}$ has a $\mathbb{Z}$-grading by setting $\deg(t_{ia})=1$, $i=1,\ldots, r$ and $a\in I_{n|n}$ since the defining relation \eqref{eq:ArlnsM} is homogeneous with respect to this grading. Similarly, there is a $\mathbb{Z}$-grading on $\mathsf{A}_{k,n}$ with $\deg(t_{jb})=1$, $j=1,\ldots, k$ and $b\in I_{n|n}$. The lemma will be proved by induction on $(\deg x, \deg y)$.

It follows from the definition of $\Theta$ in Proposition~\ref{prop:interwa} that $\deg(x_i)\leq \deg(x)$ and $\deg(y_i)\leq \deg(y)$.

The statement is true when $\deg(x)=0$ or $\deg(y)=0$ since
$$\Theta(y\otimes1)=1\otimes y,\text{ and }\Theta(1\otimes x)=x\otimes1.$$

If $\deg(x)=\deg(y)=1$, we set $x=t_{ia}$ and $y=t_{jb}$, $i=1,\ldots, r$, $j=1,\ldots,k$, and $a,b\in I_{n|n}$. We deduce from \eqref{eq:theta2} that
\begin{align*}
		\Theta(t_{jb}\otimes t_{ia})=&(-1)^{|a||b|}q^{\varphi(b,a)}t_{ia}\otimes t_{jb}
		+\delta_{b<a}\xi t_{ib}\otimes t_{ja}+\delta_{-b<a}(-1)^{|a|}\xi t_{i,-b}\otimes t_{j,-a}.
\end{align*}
In $\mathsf{A}_{r+k,n}$, the defining relation \eqref{eq:ArlnsM} yields that
\begin{align*}
	\iota_0(t_{j,b})\iota_k(t_{i,a})=&t_{j,b}t_{k+i,a}\\
	=&(-1)^{|a||b|}q^{\varphi(b,a)}t_{k+i,a}t_{j,b}
	+\xi(\delta_{b<a}-\delta_{k+i<j})t_{k+i,b}t_{j,a}\\
	&+(-1)^{|a|}\xi(\delta_{-b<a}-\delta_{k+i<-j})t_{k+i,-b}t_{j,-a}\\
	=&(-1)^{|a||b|}q^{\varphi(b,a)}\iota_k(t_{i,a})\iota_0(t_{j,b})
	+\xi\delta_{b<a}\iota_k(t_{i,b})\iota_0(t_{j,a})\\
	&+(-1)^{|a|}\xi\delta_{-b<a}\iota_k(t_{i,-b})\iota_0(t_{j,-a}).
\end{align*}
This shows that the statement is also true when $\deg(x)=\deg(y)=1$.

Now, we assume that the lemma is true for all $x,y$ with $\deg(x)\leq m$ and $\deg(y)\leq m^{\prime}$. Note that an element in $\mathsf{A}_{r,n}$ of degree $m+1$ can be written as a linear combination of $xx^{\prime}$  with $\deg(x)\leq m, \deg(x^{\prime})\leq m$, we consider such an element $xx^{\prime}$ in $\mathsf{A}_{r,n}$ and $y\in\mathsf{A}_{k,n}$ with $\deg(y)\leq m^{\prime}$.

Suppose that
\begin{align*}
\Theta(y\otimes x)=\sum_i x_i\otimes y_i,\text{ and }\Theta(y_i\otimes x^{\prime})=\sum_j x_{ij}^{\prime}\otimes y_{ij}, \text{ for each }j,
\end{align*}
where $x_i, x_{ij}^{\prime}\in\mathsf{A}_{r,n}$ and $y_i, y_{ij}\in\mathsf{A}_{k,n}$. Since $\deg(x)\leq m$ and $\deg y\leq m^{\prime}$, the hypothesis implies that
\begin{align}
\iota_0(y)\iota_k(x)=\sum_i\iota_k(x_i)\iota_0(y_i),\text{ and }
\iota_0(y_i)\iota_k(x^{\prime})=\sum_j\iota_k(x_{ij}^{\prime})\iota_0(y_{ij})\text{ for each }j.
\label{eq:indhy}
\end{align}

By the commutative diagram \eqref{eq:intwass1} for $\Theta$, we have
\begin{align*}
\Theta(y\otimes xx^{\prime})
=&\Theta\circ(1\otimes\mathrm{mul})(y\otimes x\otimes x^{\prime})
=(\mathrm{mul}\otimes1)\circ(1\otimes\Theta)\circ(\Theta\otimes1)(y\otimes x\otimes x^{\prime})\\
=&\sum_i(\mathrm{mul}\otimes1)\circ(1\otimes\Theta)(x_i\otimes y_i\otimes x^{\prime})
=\sum_{i,j}x_ix_{ij}^{\prime}\otimes y_{ij}.
\end{align*}
Note that both $\iota_0$ and $\iota_k$ are homomorphisms of associative superalgebras, we verify by using \eqref{eq:indhy} that
\begin{align*}
\sum_{i, j}\iota_k(x_ix_{ij}^{\prime})\iota_0(y_{ij})
=&\sum_{i,j}\iota_k(x_i)\iota_k(x_{ij}^{\prime})\iota_0(y_{ij})
=\sum_i\iota_k(x_i)\iota_0(y_i)\iota_k(x^{\prime})\\
=&\iota_0(y)\iota_k(x)\iota_k(x^{\prime})
=\iota_0(y)\iota_k(xx^{\prime}).
\end{align*}
This shows that the lemma is also true for $xx^{\prime}$ and $y$.

Similarly, we show that the lemma is true for $x$ and $yy^{\prime}$ for $x\in\mathsf{A}_{r,n}$ and $y,y^{\prime}\in\mathsf{A}_{k,n}$ provided that $\deg(x)\leq m$ and $\deg(y), \deg(y^{\prime})\leq m^{\prime}$. Then the lemma is true for all $x\in\mathsf{A}_{r,n}$ and $y\in\mathsf{A}_{k,n}$ by induction.
\end{proof}

\begin{theorem}
\label{thm:isoA}
The $\mathbb{C}(q)$-linear map
\begin{equation*}
\sigma:\mathsf{A}_{r,n}\otimes\mathsf{A}_{k,n}
\rightarrow\mathsf{A}_{r+k,n}, \quad x\otimes y\mapsto \iota_k(x)\iota_0(y),\quad x\in\mathsf{A}_{r,n},\quad y\in\mathsf{A}_{k,n}
\end{equation*}
is an isomorphism of\, $\Uq(\mathfrak{q}_n)$-supermodule superalgebras, where \,$\iota_k:\mathsf{A}_{r,n}\rightarrow\mathsf{A}_{r+k,n}$ and $\iota_0:\mathsf{A}_{k,n}\rightarrow\mathsf{A}_{r+k,n}$ are the maps given in Lemma~\ref{lem:embed}.
\end{theorem}
\begin{proof}
We first prove that $\sigma$ is a homomorphism of associative superalgebras. By Lemma~\ref{lem:embed}, $\iota_k:\mathsf{A}_{r,n}\rightarrow\mathsf{A}_{r+k,n}$ and $\iota_0:\mathsf{A}_{k,n}\rightarrow\mathsf{A}_{r+k,n}$ are homomorphisms of associative superalgebras. For $x, x^{\prime}\in\mathsf{A}_{r,n}$ and $y, y^{\prime}\in\mathsf{A}_{k,n}$, we have
\begin{align*}
\sigma\left((x^{\prime}\otimes1)(x\otimes y)\right)
=&\sigma(x^{\prime}x\otimes y)=\iota_0(x^{\prime}x)\iota_k(y)
=\iota_0(x^{\prime})\iota_0(x)\iota_r(y)\\
=&\sigma(x^{\prime}\otimes1)\sigma(x\otimes y).\\
\sigma\left((x\otimes y)(1\otimes y^{\prime})\right)
=&\sigma(x\otimes yy^{\prime})=\iota_0(x)\iota_r(yy^{\prime})
=\iota_0(x)\iota_r(y)\iota_r(y^{\prime})\\
=&\sigma(x\otimes y)\sigma(1\otimes y^{\prime}).
\end{align*}

Note that $(1\otimes y)(x\otimes1)=\Theta(y\otimes x)$ in the braided tensor product $\mathsf{A}_{r,n}\otimes\mathsf{A}_{k,n}$. We assume that
$$\Theta(y\otimes x)=\sum_i x_i\otimes y_i,$$
where $x_i\in\mathsf{A}_{r,n}$ and $y_i\in\mathsf{A}_{k,n}$. It follows from Lemma~\ref{lem:braidmul} that
$$\sigma((1\otimes y)(x\otimes1))
=\sum_i\iota_k(x_i)\iota_0(y_i)=\iota_0(y)\iota_k(x)=\sigma(1\otimes y)\sigma(x\otimes1).$$
Hence, $\sigma$ is a homomorphism of associative superalgebras.

Since $\iota_k:\mathsf{A}_{r,n}\rightarrow\mathsf{A}_{r+k,n}$ and $\iota_0:\mathsf{A}_{k,n}\rightarrow\mathsf{A}_{r+k,n}$ are both homomorphisms of $\Uq(\mathfrak{q}_n)$-supermodules and $\mathsf{A}_{r+k,n}$ is a $\Uq(\mathfrak{q}_n)$-supermodule superalgebra, we have
\begin{align*}
\sigma\left(\Phi_u(x\otimes y)\right)
=&\sum\limits_{(u)}(-1)^{|u_{(2)}||x|}\sigma\left(\Phi_{u_{(1)}}(x)\otimes\Phi_{u_{(2)}}(y)\right)\\
=&\sum\limits_{(u)}(-1)^{|u_{(2)}||x|}\iota_k\left(\Phi_{u_{(1)}}(x)\right)\iota_0\left(\Phi_{u_{(2)}}(y)\right)\\
=&\sum\limits_{(u)}(-1)^{|u_{(2)}||x|}\Phi_{u_{(1)}}\left(\iota_k(x)\right)\Phi_{u_{(2)}}\left(\iota_0(y)\right)\\
=&\Phi_u\left(\iota_k(x)\iota_0(y)\right)\\
=&\Phi_u(\sigma(x\otimes y)).
\end{align*}
This shows that $\sigma$ is a homomorphism of $\Uq(\mathfrak{q}_n)$-supermodules.

The homomorphism $\sigma$ is surjective since $\mathsf{A}_{r+k, n}$ is generated by $t_{ia}$ for $i=1,\ldots,r+k, a\in I_{n|n}$. While the injectivity of $\sigma$ follows from the fact  that a $\mathbb{C}(q)$-basis of $\mathsf{A}_{r,n}$ (resp. $\mathsf{A}_{k,n}$) is given by monomials
$$\prod\limits_{i=1,\ldots,r, a\in I_{n|n}}t_{ia}^{d_{ia}},\quad \left(\text{resp}. \prod\limits_{i=1,\ldots,k, a\in I_{n|n}}t_{ia}^{d_{ia}} \right)$$
where $d_{ia}\geq0$ if $|a|=\bar{0}$ and $d_{ia}\in{0,1}$ if $|a|=\bar{1}$, the product is taken with respect to the lexicographic order (see \cite[Proposition~3.6.4]{BDK20}).
\end{proof}

Given two quantum coordinate superalgebras $\mathsf{A}_{r,n}$ and $\mathsf{A}_{k,n}$, the braided tensor product $\mathsf{A}_{r,n}\otimes\mathsf{A}_{k,n}$ is also a $\Uq(\mathfrak{q}_n)$-supermodule superalgebra, which is isomorphic to $\mathsf{A}_{r+k,n}$ by Theorem~\ref{thm:isoA}. Hence, a further braided tensor product $\left(\mathsf{A}_{r,n}\otimes\mathsf{A}_{k,n}\right)\otimes\mathsf{A}_{p,n}$ is permitted. It is also a $\Uq(\mathfrak{q}_n)$-supermodule superalgebra isomorphic to $\mathsf{A}_{r+k+p,n}$. Consequently, the braided tensor product of quantum coordinate superalgebras of type $Q$ is associative.

\begin{corollary}
There is an isomorphism of $\Uq(\mathfrak{q}_n)$-supermodule superalgebras
$$\left(\mathsf{A}_{r,n}\otimes\mathsf{A}_{k,n}\right)\otimes\mathsf{A}_{p,n}\cong\mathsf{A}_{r,n}\otimes\left(\mathsf{A}_{k,n}\otimes\mathsf{A}_{p,n}\right).\eqno{\qed}$$
\end{corollary}

\begin{corollary}
\label{cor:isoArcopy}
The braided tensor product $\mathsf{A}_{1,n}^{\otimes r}$ of $r$-copies of $\mathsf{A}_{1,n}$ is isomorphic to $\mathsf{A}_{r,n}$ as $\Uq(\mathfrak{q}_n)$-supermodule superalgebras.\qed
\end{corollary}

\begin{remark}\label{IsoQQS}
 The quantum coordinate superalgebra $\mathsf{A}_{1,n}$ is isomorphic to quantum supersymmetric algebra $\mathrm{S}_q(V)$, that is a unital associative superalgebra presented by generators $v_a$ for $a\in I_{n|n}$, and relation:
\begin{align*}
qv_av_b=q^{\varphi(a,b)}(-1)^{|a||b|}v_bv_a+\delta_{a<b}\xi v_av_b+\delta_{-a<b}\xi(-1)^{|b|}v_{-a}v_{-b},\quad a,b\in I_{n|n}.
\end{align*}
By \cite[Proposition~3.6.4]{BDK20}, there is a parity preserving vector superspace isomorphism between $\mathsf{A}_{r,n}$ and $\mathrm{S}_q(V)^{\otimes r}$. According to Corollary~\ref{cor:isoArcopy}, $\mathrm{S}_q(V)^{\otimes r}$ is understood as an associative superalgebra obtained by the braided tensor product (see \cite{BZ08} for the braided symmetric algebra in the case of a quantum group). It is isomorphic to $\mathsf{A}_{r,n}$ as $\Uq(\mathfrak{q}_n)$-supermodule superalgebras. Consequently, the quantum coordinate superalgebra $\mathsf{A}_{r,n}$ (or $\mathrm{S}_q(V)^{\otimes r}$) is a flat deformation of supersymmetric algebra $\mathrm{Sym}\left(V^{\oplus r}\right)$. 
\end{remark}
\bigskip

Next, we construct a quantum analogue of $\mathrm{Sym}\left(\Pi(V)^{\oplus k}\right)$, where $\Pi$ is the parity reversing functor. For any $\mathrm{U}_q(\mathfrak{q}_n)$-supermodule $M=M_{\bar{0}}\oplus M_{\bar{1}}$, define
\begin{equation*}
\Pi(M)=\Pi(M)_{\bar{0}}\oplus \Pi(M)_{\bar{1}},\quad \Pi(M)_{i}=M_{i+\bar{1}},~\forall i\in\mathbb{Z}_2.
\end{equation*}
Then $\Pi(M)$ is also a $\mathrm{U}_q(\mathfrak{q}_n)$-supermodule with the same action.

\begin{definition}
We define $\mathsf{A}_{k,n}^{\Pi}$ to be the unital associative superalgebra presented by the generators $t_{ia}^{\pi}$ of parity $|a|+\bar{1}$ for $i=1,\ldots,k$ and $a\in I_{n|n}$ and the defining relation
\begin{equation}
R_+^{12}\check{T}_+^{1[3]}\check{T}_+^{2[3]}=\check{T}_+^{2[3]}\check{T}_+^{1[3]}S_J^{12},
\label{eq:ApiRTT}
\end{equation}
where $\check{T}_+=\sum\limits_{i=1}^k\sum\limits_{a\in I_{n|n}}E_{ia}\otimes t_{ia}^{\pi}$, the tensor matrix $R_+$ is given in \eqref{eq:rmatrix},
\begin{equation}
\begin{aligned}
S_{J}=&(1\otimes J)S(1\otimes J)\\
=&-\sum_{a,b\in I_{n|n}}q^{-\varphi(a,b)}E_{aa}\otimes E_{bb}+\xi\sum\limits_{b<a}(-1)^{|a|}(E_{ba}+E_{-b,-a})\otimes E_{ab}
,
\end{aligned}
\label{eq:stilde}
\end{equation}
and $J=\sum\limits_{a\in I_{n|n}}(-1)^{|a|}E_{-a,a}$.
\end{definition}

\begin{proposition}
$\mathsf{A}_{k,n}^{\Pi}$ is a $\Uq(\mathfrak{q}_n)$-supermodule superalgebra with the $\Uq(\mathfrak{q}_n)$-action $\Phi^{\pi}$ determined by 
\begin{equation}
L^{[2]3}\underset{\Phi^{\pi}}{\cdot}\check{T}_+^{1[2]}=\check{T}_+^{1[2]}S^{13}.\label{eq:PhiPi}
\end{equation}
\end{proposition}
\begin{proof}
Let $W\subset\mathsf{A}_{k,n}^{\Pi}$ be the $\mathbb{C}(q)$-subspace spanned by $t_{ia}^{\pi}$ for $i=1,\ldots, k$ and $a\in I_{n|n}$. We first show that \eqref{eq:PhiPi} defines a $\Uq(\mathfrak{q}_n)$-supermodule structure on $W$. Note that $L$ is the generator matrix of $\Uq(\mathfrak{q}_n)$, we verify that
\begin{align*}
\left(L^{[2]3}L^{[2]4}S^{34}\right)\underset{\Phi^{\pi}}{\cdot}\check{T}_+^{1[2]}
=&L^{[2]3}\underset{\Phi^{\pi}}{\cdot}\left(L^{[2]4}\underset{\Phi^{\pi}}{\cdot}\check{T}_+^{1[2]}\right)S^{34}
=L^{[2]3}\underset{\Phi^{\pi}}{\cdot}\check{T}_+^{1[2]}S^{14}S^{34}\\
=&\check{T}_+^{1[2]}S^{13}S^{14}S^{34},\\
\left(S^{34}L^{[2]4}L^{[2]3}\right)\underset{\Phi^{\pi}}{\cdot}\check{T}_+^{1[2]}
=&S^{34}L^{2[4]}\underset{\Phi^{\pi}}{\cdot}\left(L^{[2]3}\underset{\Phi^{\pi}}{\cdot}\check{T}_+^{1[2]}\right)
=S^{34}L^{2[4]}\underset{\Phi^{\pi}}{\cdot}\check{T}_+^{1[2]}S^{13}\\
=&S^{34}\check{T}_+^{1[2]}S^{14}S^{13}
=\check{T}_+^{1[2]}S^{34}S^{14}S^{13}.
\end{align*}
Since $S$ satisfies the quantum Yang-Baxter equation, \eqref{eq:PhiPi} determines a well-defined $\Uq(\mathfrak{q}_n)$-action on $W$.

The $\Uq(\mathfrak{q}_n)$-action is extended to the free associative superalgebra generated by $t_{ia}^{\pi}$, $i=1,\ldots, k$ and $a\in I_{n|n}$ according to the definition of a $\Uq(\mathfrak{q}_n)$-supermodule superalgebra:
\begin{equation}
\Phi^{\pi}_u(xy)=\sum\limits_{(u)}(-1)^{|u_{(2)}||x|}\Phi^{\pi}_{u_{(1)}}(x)\Phi^{\pi}_{u_{(2)}}(y),\quad u\in\Uq(\mathfrak{q}_n), x,y\in\mathsf{A}_{k,n}^{\pi}.\label{eq:modalg}
\end{equation}
It suffices to show that the action $\Phi^{\pi}_{L_{ab}}$ preserves the defining relation \eqref{eq:ApiRTT} of $\mathsf{A}_{k,n}^{\Pi}$ for $a,b\in I_{n|n}$.

By \eqref{eq:modalg}, we have
\begin{align*}
L^{[3]4}\underset{\Phi^{\pi}}{\cdot}\left(\check{T}_+^{1[3]}\check{T}_+^{2[3]}\right)
=\left(L^{[3]4}\underset{\Phi^{\pi}}{\cdot}\check{T}_+^{1[3]}\right)
\left(L^{[3]4}\underset{\Phi^{\pi}}{\cdot}\check{T}_+^{2[3]}\right)
=\check{T}_+^{1[3]}S^{14}\check{T}_+^{2[3]}S^{24},\\
L^{[3]4}\underset{\Phi^{\pi}}{\cdot}\left(\check{T}_+^{2[3]}\check{T}_+^{1[3]}\right)
=\left(L^{[3]4}\underset{\Phi^{\pi}}{\cdot}\check{T}_+^{2[3]}\right)
\left(L^{[3]4}\underset{\Phi^{\pi}}{\cdot}\check{T}_+^{1[3]}\right)
=\check{T}_+^{2[3]}S^{24}\check{T}_+^{1[3]}S^{14}.
\end{align*}
Then the defining relation \eqref{eq:ApiRTT} implies that
\begin{align*}
L^{[3]4}\underset{\Phi^{\pi}}{\cdot}\left(R_+^{12}\check{T}_+^{1[3]}\check{T}_+^{2[3]}\right)
=&R_+^{12}\check{T}_+^{1[3]}S^{14}\check{T}_+^{2[3]}S^{24}
=\check{T}_+^{2[3]}\check{T}_+^{1[3]}S_J^{12}S^{14}S^{24},\\
L^{[3]4}\underset{\Phi^{\pi}}{\cdot}\left(\check{T}_+^{2[3]}\check{T}_+^{1[3]}S_J^{12}\right)
=&\check{T}_+^{2[3]}\check{T}_+^{1[3]}S^{24}S^{14}S_J^{12}.
\end{align*}
Note that $S_J=(1\otimes J)S(1\otimes J)$ and $(J\otimes 1)S=S(J\otimes 1)$, we have
\begin{align*}
S_J^{12}S^{14}S^{24}=&J^{\underline{2}}S^{12}J^{\underline{2}}S^{14}S^{24}=J^{\underline{2}}S^{12}S^{14}S^{24}J^{\underline{2}}\\
=&J^{\underline{2}}S^{24}S^{14}S^{12}J^{\underline{2}}
=S^{24}S^{14}J^{\underline{2}}S^{12}J^{\underline{2}}
=S^{24}S^{14}S_J^{12},
\end{align*}
where $J^{\underline{2}}=1\otimes J\otimes 1\otimes1$. Hence, 
$$L^{[3]4}\underset{\Phi^{\pi}}{\cdot}\left(R_+^{12}\check{T}_+^{1[3]}\check{T}_+^{2[3]}\right)
=L^{[3]4}\underset{\Phi^{\pi}}{\cdot}\left(\check{T}_+^{2[3]}\check{T}_+^{1[3]}S_J^{12}\right).$$
This shows that \eqref{eq:PhiPi} determines a well-defined $\Uq(\mathfrak{q}_n)$-action on $\mathsf{A}_{k,n}^{\Pi}$, under which $\mathsf{A}_{k,n}^{\Pi}$ is a $\Uq(\mathfrak{q}_n)$-supermodule superalgebra.
\end{proof}

The defining relation \eqref{eq:ApiRTT} for $\mathsf{A}_{k,n}^{\Pi}$ can be written in terms of generators as:
\begin{align*}
q^{\delta_{ij}}t^{\pi}_{ia}t^{\pi}_{jb}-(-1)^{|a||b|}q^{-\varphi(a,b)}t^{\pi}_{jb}t^{\pi}_{ia}
=\xi\left(\delta_{b<a}-\delta_{j<i}\right)t^{\pi}_{ja}t^{\pi}_{ib}
+(-1)^{|b|}\xi\delta_{b<-a}t^{\pi}_{j,-a}t^{\pi}_{i,-b}.
\end{align*}
One observes that the classical limit of $\mathsf{A}_{k,n}^{\Pi}$ as $q\mapsto1$ is isomorphic to the associative superalgebra $\mathrm{Sym}\left(\Pi(V)^{\oplus k}\right)$. Moreover, it follows from \eqref{eq:PhiPi} that, for each fixed $i=1,\ldots,k$, the $\mathbb{C}(q)$ sub-superspace of $\mathsf{A}_{k,n}^{\Pi}$ spanned by $t_{ia}^{\pi}$ with $a\in I_{n|n}$ is a $\Uq(\mathfrak{q}_n)$-sub-supermodule that is isomorphic to $\Pi(V)$. Consequently, $\mathsf{A}_{k,n}^{\Pi}$ is a quantum analogue of $\mathrm{Sym}\left(\Pi(V)^{\oplus k}\right)$.

It has been shown in \cite[Remark~3.2]{BK21} that the $\mathrm{U}(\mathfrak{q}_n)$-supermodule superalgebra $\mathrm{Sym}\left(\Pi(V)^{\oplus k}\right)$ is isomorphic to $\mathrm{Sym}\left(V^{\oplus k}\right)$. Such an isomorphism is also valid in the quantum case.

\begin{proposition}
\label{prop:isoApi}
There is an isomorphism of $\Uq(\mathfrak{q}_n)$-supermodule superalgebras
$$\tau:\mathsf{A}_{k,n}^{\Pi}\rightarrow\mathsf{A}_{k,n},$$
such that
\begin{equation}
\tau\left(\check{T}_+\right)=T_+(J\otimes1).
\label{eq:isoAAp}
\end{equation}
\end{proposition}
\begin{proof}
We first show that \eqref{eq:isoAAp} determines a well-defined homomorphism of associative superalgebras $\mathsf{A}_{k,n}^{\Pi}\rightarrow\mathsf{A}_{k,n}$. Since $\mathsf{A}_{k,n}^{\Pi}$ is generated by entries of the generator matrix $\check{T}_+$, it suffices to verify that \eqref{eq:isoAAp} preserves the defining relation \eqref{eq:ApiRTT} for $\mathsf{A}_{k,n}^{\Pi}$.
\begin{align*}
\tau\left(R_+^{12}\check{T}_+^{1[3]}\check{T}_+^{2[3]}\right)
=&R_+^{12}T_+^{1[3]}J^{\underline{1}}T_+^{2[3]}J^{\underline{2}}
=T_+^{2[3]}T_+^{1[3]}S^{12}J^{\underline{1}}J^{\underline{2}},\\
\tau\left(\check{T}_+^{2[3]}\check{T}_+^{1[3]}S_J^{12}\right)
=&T_+^{2[3]}J^{\underline{2}}T_+^{1[3]}J^{\underline{1}}S_J^{12}
=T_+^{2[3]}T_+^{1[3]}J^{\underline{2}}J^{\underline{1}}S_J^{12},
\end{align*}
where $J^{\underline{1}}=J\otimes1\otimes1$ and $J^{\underline{2}}=1\otimes J\otimes 1$.

Since 
\begin{align*}
J^{\underline{2}}J^{\underline{1}}S_J^{12}
=J^{\underline{2}}J^{\underline{1}}J^{\underline{2}}S^{12}J^{\underline{2}}
=-J^{\underline{2}}J^{\underline{2}}S^{12}J^{\underline{1}}J^{\underline{2}}
=S^{12}J^{\underline{1}}J^{\underline{2}},
\end{align*}
we have
$$\tau\left(R_+^{12}\check{T}_+^{1[3]}\check{T}_+^{2[3]}\right)=\tau\left(\check{T}_+^{2[3]}\check{T}_+^{1[3]}S_J^{12}\right),$$
and \eqref{eq:isoAAp} determines a superalgebra homomorphism $\tau:\mathsf{A}_{k,n}^{\Pi}\rightarrow\mathsf{A}_{k,n}$.

The superalgebra homomorphism $\tau$ is an isomorphism since it has the inverse $\tau^{-1}:\mathsf{A}_{k,n}\rightarrow\mathsf{A}_{k,n}^{\Pi}$ given by
$$\tau^{-1}\left(T_+\right)=-\check{T}_+(J\otimes1).$$

Moreover, we compute that
\begin{align*}
\tau\left(L^{[2]3}\underset{\Phi^{\pi}}{\cdot}\check{T}_+^{1[2]}\right)=&\tau\left(\check{T}_+^{1[2]}S^{13}\right)=T_+^{1[2]}J^{\underline{1}}S^{13},\\
L^{[2]3}\underset{\Phi^{\pi}}{\cdot}\tau\left(\check{T}_+^{1[2]}\right)=&L^{[2]3}\underset{\Phi^{\pi}}{\cdot}T_+^{1[2]}J^{\underline{1}}=T_+^{1[2]}S^{13}J^{\underline{1}},
\end{align*}
which shows that $\tau$ is also a homomorphism of $\Uq(\mathfrak{q}_n)$-supermodules since $J^{\underline{1}}S^{13}=S^{13}J^{\underline{1}}$.
\end{proof}

\section{Dual Quantum Coordinate Superalgebras}\label{sec:dqcsa}

As a quantum analogue of the supersymmetric algebra $\mathrm{Sym}\left(V^{*\oplus s}\right)$, the \textit{dual quantum coordinate superalgebra $\bar{\mathsf{A}}_{s,n}$} was introduced in \cite[Section~2]{CW23}. This section is devoted to discussing a braided tensor product between two dual quantum coordinate superalgebras $\bar{\mathsf{A}}_{s,n}$ and $\bar{\mathsf{A}}_{l,n}$. 

\textit{The dual quantum coordinate superalgebra} $\bar{\mathsf{A}}_{s,n}$ is a unital associative superalgebra generated by elements $\bar{t}_{\alpha b}$ with $\alpha\in I_{s|s}$ and $b\in I_{n|n}$, which satisfy the following relations:
\begin{eqnarray}
\bar{t}_{\alpha,b}&=&(-1)^{|\alpha|+|b|}\bar{t}_{-\alpha,-b},\label{eq:DQCA1}\\
\bar{T}^{1[3]}\bar{T}^{2[3]}S^{12}&=&S^{12}\bar{T}^{2[3]}\bar{T}^{1[3]},\label{eq:DQCA2}
\end{eqnarray}
where $\bar{T}=\sum\limits_{\alpha\in I_{s|s},b\in I_{n|n}}(-1)^{|b|(|\alpha|+|b|)} E_{b\alpha}\otimes \bar{t}_{\alpha b}$. 

\begin{lemma}
\label{lem:DQCAP2}
An alternative presentation of the superalgebra $\bar{\mathsf{A}}_{s,n}$ is given by generators $\bar{t}_{\alpha b}$ with $\alpha=-1,\ldots,-s,  b\in I_{n|n}$ and the relation
\begin{equation}
\bar{T}_-^{1[3]}\bar{T}_-^{2[3]}R_-^{12}=S^{12}\bar{T}_-^{2[3]}\bar{T}_-^{1[3]},\label{eq:DArlnsM}
\end{equation}
where $\bar{T}_-=\sum\limits_{\alpha=-s}^{-1}\sum\limits_{b\in I_{n|n}}E_{b \alpha}\otimes \bar{t}_{\alpha b}$, 
and 
\begin{equation}
R_-=\sum\limits_{i,j=-s}^{-1}q^{-\delta_{ij}}E_{ii}\otimes E_{jj}-\xi \sum\limits_{-s\leqslant j<i\leqslant -1}E_{ji}\otimes E_{ij}
\label{eq:rnegative}
\end{equation} 
is the submatix of $S$ involving the terms $E_{ik}\otimes E_{jl}$ with $-s\leqslant i, j,k,l\leqslant -1$.
\end{lemma}

\begin{proof}
The relation \eqref{eq:DQCA2} is equivalent to
\begin{equation}
\begin{aligned}
&q^{\varphi(\alpha,\beta)}\bar{t}_{\alpha a}\bar{t}_{\beta b}-(-1)^{(|\alpha|+|a|)(|\beta|+|b|)}q^{\varphi(a,b)}\bar{t}_{\beta b}\bar{t}_{\alpha a}\\
=&\theta(\alpha,\beta,b)\xi\left(\delta_{b<a}-\delta_{\alpha<\beta}\right)\bar{t}_{\beta a}\bar{t}_{\alpha b}
+\theta(\alpha,\beta,b)(-1)^{|a|+|\beta|}\xi\left(\delta_{-\alpha<\beta}-\delta_{b<-a}\right)\bar{t}_{\beta,-a}\bar{t}_{\alpha,-b},
\end{aligned}
\label{eq:barT}
\end{equation}
for $\alpha,\beta\in I_{s|s}$ and $a,b\in I_{n|n}$. While the relation \eqref{eq:DArlnsM} is equivalent to
\begin{equation}
\begin{aligned}
&q^{-\delta_{\alpha\beta}}\bar{t}_{\alpha a}\bar{t}_{\beta b}-(-1)^{(|a|+1)(|b|+1)}q^{\varphi(a,b)}\bar{t}_{\beta b}\bar{t}_{\alpha a}\\
=&\xi\left(\delta_{\alpha<\beta}-\delta_{b<a}\right)\bar{t}_{\beta a}\bar{t}_{\alpha b}-(-1)^{|a|}\xi\delta_{b<-a}\bar{t}_{\beta,-a}\bar{t}_{\alpha,-b},
\end{aligned}\label{eq:barTpp}
\end{equation}
for $\alpha,\beta=-1,\ldots,-s$ and $a,b\in I_{n|n}$. We verify that \eqref{eq:barT} is equivalent to \eqref{eq:barTpp} provided that $\bar{t}_{\alpha, b}=(-1)^{|\alpha|+|b|}\bar{t}_{-\alpha,-b}$ for $\alpha\in I_{s|s}$ and $b\in I_{n|n}$.

The relation \eqref{eq:barTpp} is the special case of \eqref{eq:barT} where $\alpha,\beta=-1,\ldots, -s$. 
Conversely, it suffices to show the relation \eqref{eq:barTpp} also implies \eqref{eq:barT}. We consider the following four cases separately:

\textbf{Case1:} $\alpha<0, \beta<0$. In this situation, the relations \eqref{eq:barTpp} and \eqref{eq:barT} are same.

\textbf{Case2:} $\alpha>0, \beta<0$. We deduce from (\ref{eq:barTpp}) that
\begin{align*}
&q^{-\delta_{-\alpha,\beta}}\bar{t}_{-\alpha,-a}\bar{t}_{\beta b}-(-1)^{|a||b|+|a|}q^{\varphi(-a,b)}\bar{t}_{\beta b}\bar{t}_{-\alpha,-a}\\
=&\xi\left(\delta_{-\alpha<\beta}-\delta_{b<-a}\right)\bar{t}_{\beta,-a}\bar{t}_{-\alpha,b}
+(-1)^{|a|}\xi\delta_{b<a}\bar{t}_{\beta a}\bar{t}_{-\alpha,-b}.
\end{align*}
Since $\bar{t}_{-\alpha,-a}=(-1)^{|a|}\bar{t}_{\alpha a}$, we obtain that
\begin{align*}
&q^{-\delta_{-\alpha,\beta}}\bar{t}_{\alpha a}\bar{t}_{\beta b}-(-1)^{|a||b|+|a|}q^{\varphi(-a,b)}\bar{t}_{\beta b}\bar{t}_{\alpha a}\\
=&-(-1)^{|a|+|b|}\xi\left(\delta_{-\alpha<\beta}-\delta_{b<-a}\right)\bar{t}_{\beta,-a}\bar{t}_{\alpha,-b}
+(-1)^{|b|}\xi\delta_{b<a}\bar{t}_{\beta a}\bar{t}_{\alpha b},
\end{align*}
which coincides with (\ref{eq:barT}) in the situation of $\alpha>0$ and $\beta<0$.

\textbf{Case 3:} $\alpha<0$, $\beta>0$. We deduce from (\ref{eq:barTpp}) that
\begin{align*}
&q^{-\delta_{\alpha,-\beta}}\bar{t}_{\alpha a}\bar{t}_{-\beta, -b}-(-1)^{|a||b|+|b|}q^{\varphi(a,-b)}\bar{t}_{-\beta,-b}\bar{t}_{\alpha a}\\
=&\xi\left(\delta_{\alpha<-\beta}-\delta_{-b<a}\right)\bar{t}_{-\beta,a}\bar{t}_{\alpha,-b}
-(-1)^{|a|}\xi\delta_{-b<-a}\bar{t}_{-\beta,-a}\bar{t}_{\alpha b}.
\end{align*}
Then $\bar{t}_{-\beta,-b}=(-1)^{|b|}\bar{t}_{\beta b}$ implies that 
\begin{align*}
&q^{-\delta_{\alpha,-\beta}}\bar{t}_{\alpha a}\bar{t}_{\beta b}-(-1)^{|a||b|+|b|}q^{\varphi(a,-b)}\bar{t}_{\beta b}\bar{t}_{\alpha a}\\
=&-(-1)^{|a|+|b|}\xi\left(\delta_{\alpha<-\beta}-\delta_{-b<a}\right)\bar{t}_{\beta,-a}\bar{t}_{\alpha,-b}
-(-1)^{|b|}\xi\delta_{a<b}\bar{t}_{\beta a}\bar{t}_{\alpha b}.
\end{align*}
Hence,
\begin{align*}
&q^{\varphi(\alpha,\beta)}\bar{t}_{\alpha a}\bar{t}_{\beta b}-(-1)^{|a||b|+|b|}q^{\varphi(a,b)}\bar{t}_{\beta b}\bar{t}_{\alpha a}\\
=&\left(q^{-\delta_{\alpha,-\beta}}+\delta_{\alpha,-\beta}\xi\right)\bar{t}_{\alpha a}\bar{t}_{\beta b}-(-1)^{|a||b|+|b|}\left(q^{-\varphi(a,b)}+(-1)^{|b|}\left(\delta_{ab}+\delta_{a,-b}\right)\xi\right)\bar{t}_{\beta b}\bar{t}_{\alpha a}\\
=&-(-1)^{|a|+|b|}\xi\left(\delta_{\alpha<-\beta}-\delta_{-b<a}\right)\bar{t}_{\beta,-a}\bar{t}_{\alpha,-b}
-(-1)^{|b|}\xi\delta_{a<b}\bar{t}_{\beta a}\bar{t}_{\alpha b}\\
&+\delta_{\alpha,-\beta}\xi \bar{t}_{\alpha a}\bar{t}_{\beta b}-(-1)^{|a||b|}\left(\delta_{ab}+\delta_{a,-b}\right)\xi \bar{t}_{\beta b}\bar{t}_{\alpha a}\\
=&-(-1)^{|b|}\xi\left(\delta_{a<b}+\delta_{ab}\right)\bar{t}_{\beta a}\bar{t}_{\alpha b}-(-1)^{|a|+|b|}\xi\left(\delta_{\alpha<-\beta}-\delta_{-b<a}+\delta_{\alpha,-\beta}-\delta_{a,-b}\right)\bar{t}_{\beta,-a}\bar{t}_{\alpha,-b}\\
=&(-1)^{|b|}\xi \left(\delta_{b<a}-1\right)\bar{t}_{\beta a}\bar{t}_{\alpha b}+(-1)^{|a|+|b|}\xi(\delta_{-\alpha<\beta}-\delta_{b<-a})\bar{t}_{\beta,-a}\bar{t}_{\alpha,-b},
\end{align*}
which yields with (\ref{eq:barT}) in the situation of $\alpha<0$ and $\beta>0$.

\textbf{Case 4:} $\alpha>0$, $\beta>0$. We deduce from (\ref{eq:barTpp}) that
\begin{align*}
&q^{-\delta_{\alpha,\beta}}\bar{t}_{-\alpha,-a}\bar{t}_{-\beta, -b}-(-1)^{|a||b|}q^{\varphi(-a,-b)}\bar{t}_{-\beta,-b}\bar{t}_{-\alpha,-a}\\
=&\xi\left(\delta_{\beta<\alpha}-\delta_{a<b}\right)\bar{t}_{-\beta,-a}\bar{t}_{-\alpha,-b}
+(-1)^{|a|}\xi\delta_{-b<a}\bar{t}_{-\beta,a}\bar{t}_{-\alpha,b}.
\end{align*}
It follows from $\bar{t}_{-\beta,-a}=(-1)^{|a|}\bar{t}_{\beta a}$ that
\begin{align*}
q^{-\delta_{\alpha\beta}}\bar{t}_{\alpha a}\bar{t}_{\beta b}-(-1)^{|a||b|}q^{\varphi(-a,-b)}\bar{t}_{\beta b}\bar{t}_{\alpha a}
=\xi\left(\delta_{\beta<\alpha}-\delta_{a<b}\right)\bar{t}_{\beta a}\bar{t}_{\alpha b}
+(-1)^{|a|}\xi\delta_{-b<a}\bar{t}_{\beta,-a}\bar{t}_{\alpha,-b}.
\end{align*}
Hence,
\begin{align*}
&q^{\varphi(\alpha,\beta)}\bar{t}_{\alpha a}\bar{t}_{\beta b}-(-1)^{|a||b|}q^{\varphi(a,b)}\bar{t}_{\beta b}\bar{t}_{\alpha a}\\
=&\left(q^{-\delta_{\alpha\beta}}+\delta_{\alpha\beta}\xi\right)\bar{t}_{\alpha a}\bar{t}_{\beta b}
-(-1)^{|a||b|}\left(q^{\varphi(-a,-b)}+(-1)^{|b|}\left(\delta_{ab}+\delta_{a,-b}\right)\xi\right)\bar{t}_{\beta b}\bar{t}_{\alpha a}\\
=&\xi\left(\delta_{\beta<\alpha}-\delta_{a<b}\right)\bar{t}_{\beta a}\bar{t}_{\alpha b}
+(-1)^{|a|}\xi\delta_{-b<a}\bar{t}_{\beta,-a}\bar{t}_{\alpha,-b}\\
&+\delta_{\alpha\beta}\xi \bar{t}_{\alpha a}\bar{t}_{\beta b}
-(-1)^{|a||b|+|b|}\left(\delta_{ab}+\delta_{a,-b}\right)\xi \bar{t}_{\beta b}\bar{t}_{\alpha a}\\
=&\xi\left(\delta_{\beta<\alpha}-\delta_{a<b}+\delta_{\alpha\beta}-\delta_{ab}\right)\bar{t}_{\beta a}\bar{t}_{\alpha b}
+(-1)^{|a|}\xi\left(\delta_{-b<a}+\delta_{-b,a}\right)\bar{t}_{\beta,-a}\bar{t}_{\alpha,-b}\\
=&\xi(\delta_{b<a}-\delta_{\alpha<\beta})\bar{t}_{\beta a}\bar{t}_{\alpha b}+(-1)^{|a|}\xi \left(1-\delta_{b<-a}\right)\bar{t}_{\beta,-a}\bar{t}_{\alpha,-b},
\end{align*}
which yields with \eqref{eq:barT} in the situation of $\alpha>0$ and $\beta>0$. This completes the proof.
\end{proof}

\begin{remark}
The dual quantum coordinate superalgebra $\bar{\mathsf{A}}_{s,n}$ is also presented by generators $\bar{t}_{\alpha a}$ for $\alpha=1,\ldots,s$ and $a\in I_{n|n}$ and the relation
\begin{equation*}
\bar{T}_+^{1[3]}\bar{T}_+^{2[3]}R^{12}_J=S_J^{12}\bar{T}_+^{2[3]}\bar{T}_+^{1[3]},
\end{equation*}
where $\bar{T}_+=\sum\limits_{i=1}^{s}\sum\limits_{b\in I_{n|n}}(-1)^{|b|}E_{bi}\otimes \bar{t}_{i b}$, $S_J$ is the tensor matrix given in \eqref{eq:stilde}, and 
\begin{equation*}
R_J=-\sum\limits_{i,j=1}^sq^{-\delta_{ij}}E_{ii}\otimes E_{jj}+\xi \sum\limits_{1\leqslant j<i\leqslant s}E_{ji}\otimes E_{ij}
\end{equation*} 
is the submatix of $S_J$ involving the terms $E_{ik}\otimes E_{jl}$ with $1\leqslant i, j,k,l\leqslant s$.
\end{remark}

The superalgebra $\bar{\mathsf{A}}_{s,n}$ is also a $\Uq(\mathfrak{q}_n)$-supermodule superalgebra under the action $\bar{\Phi}$ given in terms of generator matrices as
\begin{equation}
L^{[2]3}\underset{\bar{\Phi}}{\cdot}\bar{T}^{1[2]}=\left(S^{-1}\right)^{13}\bar{T}^{1[2]},
\end{equation}
where 
\begin{equation}
S^{-1}=\sum\limits_{i,j\in I_{r|r}}q^{-\varphi(i,j)}E_{ii}\otimes E_{jj}-\xi \sum\limits_{i,j\in I_{r|r}, i<j }(-1)^{|i|}(E_{ji}+E_{-j,-i})\otimes E_{ij}\label{eq:sinverse}
\end{equation}
is the inverse of the tensor matrix $S$ given in \eqref{eq:smatrix}.

\begin{lemma}
	\label{lem:dembed}
For each $0\leq p\leq l$, there is an injective homomorphism of $\Uq(\mathfrak{q}_n)$-supermodule superalgebras
$$\bar{\iota}_p: \bar{\mathsf{A}}_{s,n}\rightarrow \bar{\mathsf{A}}_{s+l,n},\quad
\bar{t}_{\alpha a}\mapsto \bar{t}_{\alpha-p,a},\quad \alpha=-1,\ldots,-l,\quad a\in I_{n|n}.$$
\end{lemma}

An explicit braiding operator $\bar{\Theta}: \bar{\mathsf{A}}_{l,n}\otimes\bar{\mathsf{A}}_{s,n}
	\rightarrow\bar{\mathsf{A}}_{s,n}\otimes\bar{\mathsf{A}}_{l,n}$ can be defined as follows.

\begin{proposition}\label{prop:interbara}
	There is a unique homomorphism of $\Uq(\mathfrak{q}_n)$-supermodules
	$$\bar{\Theta}:\bar{\mathsf{A}}_{l,n}\otimes\bar{\mathsf{A}}_{s,n}
	\rightarrow\bar{\mathsf{A}}_{s,n}\otimes\bar{\mathsf{A}}_{l,n}$$
	satisfying the commutative diagrams \eqref{eq:intwass1},\eqref{eq:intwass2}, and such that 
	\begin{align}
	\bar{\Theta}\left(1\otimes \bar{x}\right)=&\bar{x}\otimes1,\quad\bar{x}\in\bar{\mathsf{A}}_{s,n},\label{eq:dtheta1a}\\
    \bar{\Theta}\left(\bar{y}\otimes1\right)=&1\otimes \bar{y}\quad \bar{y}\in\bar{\mathsf{A}}_{l,n},\label{eq:dtheta1b}\\		
	\bar{\Theta}\left(\bar{T}_l^{1[3]}\bar{T}_s^{2[4]}\right)
		=&S^{12}\bar{T}_s^{2[3]}\bar{T}_l^{1[4]},\label{eq:dtheta2}
	\end{align}
where $\bar{T}_l=\sum\limits_{\beta=-l}^{-1}\sum\limits_{b\in I_{n|n}}E_{b\beta}\otimes \bar{t}_{\beta b}$  and $\bar{T}_s=\sum\limits_{\alpha=-s}^{-1}\sum\limits_{a\in I_{n|n}}E_{a\alpha}\otimes \bar{t}_{\alpha a}$ are the generator matrices of $\bar{\mathsf{A}}_{l,n}$ and $\bar{\mathsf{A}}_{s,n}$, respectively.
\end{proposition}
\begin{proof}
The proposition follows from the same technique of the proposition \ref{prop:interwa}.
\end{proof}

The braiding operator $\bar{\Theta}$ yields a braided multiplication on $\bar{\mathsf{A}}_{s,n}\otimes\bar{\mathsf{A}}_{l,n}$:
$$\bar{\mathsf{A}}_{s,n}\otimes\bar{\mathsf{A}}_{l,n}
\otimes\bar{\mathsf{A}}_{s,n}\otimes\bar{\mathsf{A}}_{l,n}
\xrightarrow{1\otimes\bar{\Theta}\otimes1}\bar{\mathsf{A}}_{s,n}\otimes\bar{\mathsf{A}}_{s,n}\otimes\bar{\mathsf{A}}_{l,n}\otimes\bar{\mathsf{A}}_{l,n}\xrightarrow{\mathrm{mul}\otimes\mathrm{mul}}
\bar{\mathsf{A}}_{s,n}\otimes\bar{\mathsf{A}}_{l,n},
$$
under which the braided tensor product $\bar{\mathsf{A}}_{s,n}\otimes\bar{\mathsf{A}}_{l,n}$ is a $\Uq(\mathfrak{q}_n)$-supermodule superalgebra. 
In particular,
\begin{align*}
	\left(1\otimes \bar{t}_{\beta b}\right)\left(\bar{t}_{\alpha a}\otimes1\right)
	=&(-1)^{(|a|+1)(|b|+1)}q^{\varphi(b,a)}\bar{t}_{\alpha a}\otimes \bar{t}_{\beta b}
	-\delta_{a<b}\xi \bar{t}_{\alpha b}\otimes \bar{t}_{\beta a}\\
	&-(-1)^{|b|}\delta_{a<-b}\xi\bar{t}_{\alpha,-b}\otimes \bar{t}_{\beta,-a},
\end{align*}
for $\alpha=-1,\ldots,-s, \beta=-1,\ldots,-l$ and $a,b\in I_{n|n}$.

\begin{theorem}\label{thm:isodqcsa}
The $\mathbb{C}(q)$-linear map
\begin{equation*}
\bar{\sigma}:\bar{\mathsf{A}}_{s,n}\otimes\bar{\mathsf{A}}_{l,n}
\rightarrow\bar{\mathsf{A}}_{s+l,n}, \quad x\otimes y\mapsto \bar{\iota}_l(x)\bar{\iota}_0(y),\quad x\in\bar{\mathsf{A}}_{s,n},\quad y\in\bar{\mathsf{A}}_{l,n},
\end{equation*}
is an isomorphism of\, $\Uq(\mathfrak{q}_n)$-supermodule superalgebras, where $\bar{\iota}_l:\bar{\mathsf{A}}_{s,n}\mapsto\bar{\mathsf{A}}_{s+l,n}$ and $\bar{\iota}_0:\bar{\mathsf{A}}_{l,n}\mapsto\bar{\mathsf{A}}_{s+l,n}$ are the maps given in Lemma~\ref{lem:dembed}
\end{theorem}
\begin{proof}
The proof is similar to Theorem~\ref{thm:isoA}, we omit the details here.
\end{proof}

\begin{remark}
Fix $1\leq\alpha\leq s$, the $\mathbb{C}(q)$-vector sub-superspace of $\bar{\mathsf{A}}_{s,n}$ spanned by $t_{\alpha b}$ for $b\in I_{n|n}$ is isomorphic to the dual natural $\Uq(\mathfrak{q}_n)$-supermodule $V_q^*$. Hence, the dual quantum coordinate superalgebra $\bar{\mathsf{A}}_{1,n}$ is isomorphic to quantum supersymmetric algebra $\mathrm{S}_q(V^*)$.  Moreover, the superalgebra $\bar{\mathsf{A}}_{s,n}$ is isomorphic to the braided tensor product $\mathrm{S}_q(V^*)^{\otimes s}$ . 
\end{remark}

We also define a quantum analogue of the supersymmetric algebra $\mathrm{Sym}\left(\Pi(V^{*})^{\oplus l}\right)$ as follows:

\begin{definition}
$\bar{\mathsf{A}}_{l,n}^{\Pi}$ is defined as a unital associative superalgebra presented by generators $\bar{t}_{\alpha b}^{\pi}$ of parity $|b|+\bar{1}$ for $\alpha=-1,\ldots,-l$ and $b\in I_{n|n}$, and the relation
\begin{equation}
\check{\bar{T}}_-^{1[3]}\check{\bar{T}}_-^{2[3]}R_-^{12}=S_J^{12}\check{\bar{T}}_-^{2[3]}\check{\bar{T}}_-^{1[3]},
\label{eq:barApiRTT}
\end{equation}
where $\check{\bar{T}}_-=\sum\limits_{\alpha=-l}^{-1}\sum\limits_{b\in I_{n|n}}E_{b\alpha}\otimes \bar{t}_{\alpha b}^{\pi}$, the tensor matrices $R_-$ and $S_J$ are given in \eqref{eq:rnegative} and \eqref{eq:stilde}, respectively.
\end{definition}

\begin{proposition}
$\bar{\mathsf{A}}_{l,n}^{\Pi}$ is a $\Uq(\mathfrak{q}_n)$-supermodule superalgebra with the $\Uq(\mathfrak{q}_n)$-action $\bar{\Phi}^{\pi}$ determined by 
\begin{equation*}
L^{[2]3}\underset{\bar{\Phi}^{\pi}}{\cdot}\check{\bar{T}}_-^{1[2]}=\left(S^{-1}\right)^{13}\check{\bar{T}}_-^{1[2]}.\eqno{\qed}
\end{equation*}
\end{proposition}

\begin{proposition}
\label{prop:isobarApi}
There is an isomorphism of $\Uq(\mathfrak{q}_n)$-supermodule superalgebras
$$\bar{\tau}:\bar{\mathsf{A}}_{l,n}^{\Pi}\rightarrow\bar{\mathsf{A}}_{l,n},$$
such that
\begin{equation*}
\bar{\tau}\left(\check{\bar{T}}_-\right)=(J\otimes1)\bar{T}_-.
\eqno{\qed}
\end{equation*}
\end{proposition}

\section {Polynomial Invariants}\label{sec:invariant}

In the quantum case, the associative superalgebras $\mathsf{A}_{r,n}$, $\mathsf{A}_{k,n}^{\Pi}$, $\bar{\mathsf{A}}_{s,n}$ and $\bar{\mathsf{A}}_{l,n}^{\Pi}$ are quantum analogues of the supersymmetric algebras $\mathrm{Sym}\left(V^{\oplus r}\right)$, $\mathrm{Sym}\left(\Pi(V)^{\oplus k}\right)$, $\mathrm{Sym}\left(V^{*\oplus s}\right)$ and $\mathrm{Sym}\left(\Pi(V^*)^{\oplus l}\right)$, respectively. We consider the $\Uq(\mathfrak{q}_n)$-supermodule
$$\mathcal{B}^{r,k}_{s,l}:=\mathsf{A}_{r,n}\otimes\mathsf{A}_{k,n}^{\Pi}\otimes\bar{\mathsf{A}}_{s,n}\otimes\bar{\mathsf{A}}_{l,n}^{\Pi}.$$

By Proposition~\ref{prop:isoApi}, $\mathsf{A}_{k,n}^{\Pi}$ is isomorphic to $\mathsf{A}_{k,n}$. Then $\mathsf{A}_{r,n}\otimes\mathsf{A}_{k,n}^{\Pi}$ is a $\Uq(\mathfrak{q}_n)$-supermodule superalgebra under the braided multiplication. It follows from Theorem~\ref{thm:isoA} that $\mathsf{A}_{r,n}\otimes\mathsf{A}_{k,n}^{\Pi}$ is also isomorphic to $\mathsf{A}_{r+k,n}$ as $\Uq(\mathfrak{q}_n)$-supermodule superalgebras.

Recall from \cite[Section~2]{CW23} that 
the associative superalgbera $\mathsf{A}_{r+k,n}$ is also a $\Uq(\mathfrak{q}_{r+k})^{\mathrm{cop}}$-supermodule superalgebra with the action $\Psi$ determined by
\begin{equation}
L^{[2]3}\underset{\Psi}{\cdot}T^{1[2]}=\left(S^{-1}\right)^{13}T^{1[2]},
\end{equation}
where $L$ is the generator matrix of $\Uq(\mathfrak{q}_{r+k})$, $T$ is the generator matrix of $\mathsf{A}_{r+k,n}$, and $S^{-1}$ is given in \eqref{eq:sinverse}. Moreover, the $\Uq(\mathfrak{q}_{r+k})$-action $\Psi$ and the $\Uq(\mathfrak{q}_n)$-action $\Phi$ on $\mathsf{A}_{r+k,n}$ are super-commutative. Thus $\mathsf{A}_{r+k,n}$ is a $\Uq(\mathfrak{q}_{r+k})\otimes \Uq(\mathfrak{q}_n)$-supermodule. Howe duality \cite[Theorem~4.2]{CW20} implies that $\mathsf{A}_{r+k,n}$ admits a multiplicity-free decomposition as a $\Uq(\mathfrak{q}_{r+k})\otimes \Uq(\mathfrak{q}_n)$-supermodule.

Similarly, we deduce from Theorem~\ref{thm:isodqcsa} and Proposition~\ref{prop:isobarApi} that the braided tensor product $\bar{\mathsf{A}}_{s,n}\otimes\bar{\mathsf{A}}_{l,n}^{\Pi}$  is isomorphic to $\bar{\mathsf{A}}_{s+l,n}$ as $\Uq(\mathfrak{q}_n)$-supermodule superalgebras. 

Besides being a $\Uq(\mathfrak{q}_n)$-supermodule, the superalgebra $\bar{\mathsf{A}}_{s+l,n}$ is a $\Uq(\mathfrak{q}_{s+l})$-supermodule under the action $\bar{\Psi}$ given by
\begin{equation}
L^{[2]3}\underset{\bar{\Psi}}{\cdot}\bar{T}^{1[2]}=\bar{T}^{1[2]} \tilde{S}^{13},
\end{equation}
where $L$ is the generator matrix of $\Uq(\mathfrak{q}_{s+l})$, $\bar{T}$ is the generator matrix of $\bar{\mathsf{A}}_{s+l,n}$, and $\tilde{S}=(1\otimes D)S(1\otimes D^{-1})$ with $D=\sum\limits_{\alpha=1}^{s+l}q^{2\alpha}(E_{\alpha\alpha}+E_{-\alpha,-\alpha})$. 
The $\Uq(\mathfrak{q}_{s+l})$-action $\bar{\Psi}$ on $\bar{\mathsf{A}}_{s+l,n}$ super-commutes with the $\Uq(\mathfrak{q}_n)$-action $\bar{\Phi}$.  As a $\Uq(\mathfrak{q}_n)\otimes\Uq(\mathfrak{q}_{s+l})$-supermodule, $\bar{\mathsf{A}}_{s+l,n}$ also admits a multiplicity-free decomposition.

Using the braided tensor product $\mathcal{B}^{r+k,0}_{s+l,0}=\mathsf{A}_{r+k,n}\otimes\bar{\mathsf{A}}_{s+l,n}$ introduced in \cite[Section~4]{CW23}, the $\Uq(\mathfrak{q}_n)$-supermodule $\mathcal{B}^{r,k}_{s,l}$ is also an associative superalgebra with respect to the braided multiplication. Moreover, it is a $\Uq(\mathfrak{q}_n)$-supermodule superalgebra. By \cite[Theorem~5.10]{CW23}, we conclude that
\begin{theorem}
The $\Uq(\mathfrak{q}_n)$-invariant sub-superalgebra of $\mathcal{B}^{r,k}_{s,l}$ 
$$\left(\mathcal{B}^{r,k}_{s,l}\right)^{\Uq(\mathfrak{q}_n)}:=\left\{z\in\mathcal{B}^{r,k}_{s,l}|\Phi_u(z)=\varepsilon(u)z, \forall u\in\Uq(\mathfrak{q}_n)\right\}$$
is generated by the following elements:
\begin{align*}
x_{i\alpha}=&\sum_{p\in I_{n|n}}t_{ip}\otimes1\otimes \bar{t}_{\alpha p}\otimes1,&
y_{j\alpha}=&\sum_{p\in I_{n|n}}1\otimes t_{j,-p}^{\pi}\otimes \bar{t}_{\alpha p}\otimes1,\\
z_{i\beta}=&\sum_{p\in I_{n|n}}(-1)^{|p|}t_{ip}\otimes1\otimes 1\otimes\bar{t}_{\beta,-p}^{\pi},&
w_{j\beta}=&\sum_{p\in I_{n|n}}(-1)^{|p|}1\otimes t_{jp}^{\pi}\otimes 1\otimes\bar{t}_{\beta p}^{\pi},
\end{align*}
for $i=1,\ldots, r$, $j=1,\ldots, k$, $\alpha=1,\ldots,s$ and $\beta=-1,\ldots, -l$.\qed
\end{theorem}

\section*{Acknowledgments}
The project is supported by the National Natural Science Foundation of China (Nos. 12071150 and 12071026), the Science and Technology Planning Project of Guangzhou (No. 202102021204).

\end{document}